\documentclass[11pt,reqno,oneside]{amsart}

\usepackage[asymmetric,top=3.5cm,bottom=4.3cm,left=3.1cm,right=3.1cm]{geometry}
\usepackage{amsmath,amsfonts,amsthm,mathrsfs,amssymb,cite}
\usepackage[usenames]{color}

\usepackage{mathtools} 
\usepackage{graphics}
\usepackage{framed}

\usepackage{enumerate}

\usepackage{hyperref}
\usepackage{xcolor}
\hypersetup{
  colorlinks,
  linkcolor={blue!80!black},
  urlcolor={blue!80!black},
  citecolor={blue!80!black}
}
\usepackage[capitalize,noabbrev]{cleveref}

\usepackage{graphicx}
\usepackage{latexsym}
\usepackage{enumerate}

\newtheorem{thm}{Theorem}[section]
\newtheorem{cor}[thm]{Corollary}

\newtheorem{Lemma}[thm]{Lemma}
\numberwithin{equation}{section}
\newtheorem{remark}[thm]{Remark}

\newcommand{\Bx}{\mathbf{x}}

\newcommand{\CR}{\mathcal{R}}

\newcommand{\RR}{\mathbb{R}}

\author{}
\date{Nov 2020}
\title[Vanishing near corners of conductive transmission eigenfunctions]{On vanishing near corners of conductive transmission eigenfunctions}

\author{Youjun Deng}
\address{School of Mathematics and Statistics, Central South University, Changsha, Hunan, China}
\email{youjundeng@csu.edu.cn; dengyijun\_001@163.com}

\author{Chaohua Duan}
\address{School of Mathematics and Statistics, Central South University, Changsha, Hunan, China}
\email{chaohua\underline{ }duan@163.com}

\author{Hongyu Liu}
\address{Department of Mathematics, City University of Hong Kong, Kowloon, Hong Kong, China.}
\email{hongyu.liuip@gmail.com, hongyliu@cityu.edu.hk}

\allowdisplaybreaks

\begin{document}
\maketitle

\begin{abstract}

In this paper, we consider the transmission eigenvalue problem associated with a general conductive transmission condition and study the geometric structures of the transmission eigenfunctions. We prove that under a mild regularity condition in terms of the Herglotz approximations of one of the pair of the transmission eigenfunctions, the eigenfunctions must be vanishing around a corner on the boundary. The Herglotz approximation can be regarded as the Fourier transform of the transmission eigenfunction in terms of the plane waves, and the growth rate of the transformed function can be used to characterize the regularity of the underlying wave function. The geometric structures derived in this paper include the related results in \cite{ref1,ref3} as special cases and verify that the vanishing around corners is a generic local geometric property of the transmission eigenfunctions.

    \medskip

\noindent{\bf Keywords:}~~ conductive transmission eigenfunctions; corner singularity; geometric structures; vanishing; Herglotz approximation.

\noindent{\bf 2010 Mathematics Subject Classification:}~~58J05, 35P25 (primary); 35Q60, 78A05 (secondary).


    \end{abstract}

\section{Introduction}

\subsection{Background}
In its general form, the transmission eigenvalue problem is given as follows (cf. \cite{Liu}):
\begin{equation}\label{eq:trans1}
\mathcal{P}_1(\mathbf{x}, D)u=-\lambda u, \quad \mathcal{P}_2(x, D)v=-\lambda v\ \ \mbox{in}\ \ \Omega; \quad \mathcal{C}(u)=\mathcal{C}(v)\ \ \mbox{on}\ \ \partial\Omega,
\end{equation}
where $\Omega$ is a bounded Lipschitz domain in $\mathbb{R}^n$, $n=2,3$, with a connected complement $\mathbb{R}^n\backslash\overline{\Omega}$ and $P_j(\mathbf{x}, D)$ are two elliptic partial differential operators (PDOs) with $D$ signifying the differentiations with respect to $\mathbf{x}=(x_j)_{j=1}^n\in\mathbb{R}^n$, and $\mathcal{C}$ denotes the Cauchy data set. If there exists a nontrivial pair of solutions $(u, v)$, then $\lambda\in\mathbb{C}$ is called a transmission eigenvalue and $(u, v)$ are the corresponding pair of transmission eigenfunctions.

Though the PDOs $\mathcal{P}_j$, $j=1,2$, are generally elliptic, selfadjoint and linear,  the transmission eigenvalue problems of the form \eqref{eq:trans1} are a type of non-elliptic, non-selfadjoint and nonlinear (in terms of the transmission eigenvalue $\lambda$) spectral problems, making the corresponding spectral study highly intriguing and challenging; see \cite{Liu} for some related discussion. The transmission eigenvalue problems arise in the wave scattering theory and connect to many aspects of the wave scattering theory in a delicate way. Indeed, many of the spectral results established for the transmission eigenvalue problems in the literature have found important applications in the wave scattering theory, including generating novel wave imaging and sensing schemes, producing important implications to invisibility cloaking and proving new uniqueness results for inverse scattering problems. We refer to \cite{CCH,CH,CKreview,Liu} for historical accounts and surveys on the state-of-the-art developments of the spectral studies for the transmission eigenvalue problems in the literature.

To a great extent, the spectral properties of the (real) transmission eigenvalues resemble those for the classical Dirichlet/Neumann Laplacian: there are infinitely many real transmission eigenvalues which are discrete and accumulate only at infinity. Nevertheless, due to the non-selfadjointness, there are complex transmission eigenvalues; see \cite{CCH,CKreview} the references cited therein. Recently, several local and global geometric structures of distinct features were discovered for the transmission eigenfunctions \cite{ref2,BLin,BLLW,ref3,BL2,BL3,BL4,BXL,CX,CDL,CDHLW1,ref1} and all of them have produced interesting applications of practical importance in the scattering theory. In this paper, we are concerned with the vanishing property of the transmission eigenfunctions around a corner on the boundary of the domain, which was first discovered in \cite{ref3} and further investigated in \cite{ref1}. Before discussing our major discoveries, we next specify the transmission eigenvalue problem as well as its vanishing properties in our study.

Let $\Omega$ be a bounded Lipschitz domain in $\mathbb{R}^{n}, n=2,3,$, with a connected complement $\mathbb{R}^n\backslash\overline{\Omega}$, and $V \in L^{\infty}(\Omega)$ and $\eta \in L^{\infty}(\partial \Omega)$ be possibly complex-valued functions. Consider the following transmission eigenvalue problem for $v, w \in H^{1}(\Omega)$ and $\lambda=k^2$, $k\in\mathbb{R}_+$:
\begin{equation}\label{1.1}
\begin{cases}
\big(\Delta+k^2(1+V)\big) w=0\ &\ \mbox{in}\ \ \Omega,\medskip\\
(\Delta+k^2) v=0\ &\ \mbox{in}\ \ \Omega,\medskip\\
w=v,\ \ \partial_\nu w=\partial_\nu v+\eta v\ &\ \mbox{on}\ \partial\Omega,
\end{cases}
\end{equation}
where $\nu\in\mathbb{S}^{n-1}$ signifies the exterior unit normal to $\partial\Omega$. Two remarks concerning the formulation of the transmission eigenvalue problem \eqref{1.1} are in order. First, we introduce $k^2$ to denote the transmission eigenvalue. On the one hand, $k$ signifies a wavenumber in the physical setup and on the other hand, this notation shall ease the exposition of our subsequent mathematical arguments. Though only $k\in\mathbb{R}_+$ is physically meaningful, some of our subsequent results also hold for the case that $k$ is a complex number, which should be clear from the context. Second, the second transmission condition on $\partial\Omega$ in \eqref{1.1} is known as the conductive transmission condition. This type of transmission condition arises in modelling wave interaction with a certain material object and can find important applications in magnetotellurics; see e.g. \cite{CDL,ref1} and the references cited therein for more relevant physical backgrounds. On the other hand, if one simply takes $\eta\equiv 0$, \eqref{1.1} is reduced to the transmission eigenvalue problem that has been more intensively studied in the literature. In order to signify such a generalization and extension, we refer to the eigenvalue problem \eqref{1.1} as the \emph{conductive transmission eigenvalue problem}, which includes the conventional transmission eigenvalue problem as a special case.

Let $\mathbf{x}_c\in\partial\Omega$ be a corner point, which shall be made more precise in what follows. Let $B_\rho(\mathbf{x}_c)$ denote a ball of radius $\rho\in\mathbb{R}_+$ centred at $\mathbf{x}_c$. The vanishing property of the transmission eigenfunction is described as follows:
\begin{equation}\label{eq:v1}
    \lim _{\rho \rightarrow+0} \frac{1}{{m}\left(B\left(\Bx_{c}, \rho\right)\cap \Omega\right)} \int_{B\left(\Bx_{c}, \rho\right)\cap\Omega}|\psi(\Bx)|\, \mathrm{d} \Bx=0,\ \ \psi=w\ \ \mbox{or}\ \ v,
\end{equation}
where ${m}$ denotes the Lebesgue measure. It is noted that $w$ and $v$ are $H^1$-functions and the vanishing at a boundary point should be understood in the integral sense. On the other hand, if $\psi$ is a continuous function in a neighbourhood of $\mathbf{x}_c$, \eqref{eq:v1} clearly implies that $\psi(\mathbf{x}_c)=0$. In fact, the regularity of the transmission eigenfunctions $w$ and $v$ in \eqref{1.1} is critical for the establishment of the vanishing property \eqref{eq:v1}. Under the regularity condition that both $w$ and $v$ are additionally  H\"older continuous, namely $C^\alpha$ continuous with $\alpha\in (0, 1)$, it is shown in \cite{ref3} and \cite{ref1} that the vanishing property holds respectively in the cases with $\eta\equiv 0$ and $\eta\neq 0$. By the classical result on the quantitative behaviours of the solutions to elliptic PDEs around a corner (cf. \cite{Cos,Dauge88,Grisvard}), we have the following decompositions
\begin{equation}\label{eq:decom1}
w=w_{\mathrm{singular}}+w_{\mathrm{regular}},\quad v=v_{\mathrm{singular}}+v_{\mathrm{regular}},
\end{equation}
where the regular parts belong to $H^2$ and hence by the standard Sobolev embedding, they are H\"older continuous. The singular parts may also be H\"older continuous provided the coefficients, namely $V$, as well as the boundary data of $w$ and $v$ around the corner are sufficiently regular. However, in the transmission eigenvalue problem \eqref{1.1}, the boundary data, namely $(w|_{\partial\Omega}, \partial_\nu w|_{\partial\Omega})$ and $(v|_{\partial\Omega}, \partial_\nu v|_{\partial\Omega})$ are not specified. Hence, it may happen that the transmission eigenfunctions are $H^1$ but not H\"older continuous. Clearly, according to our discussion above, the vanishing property may serve as an indicator for such singular behaviours of the transmission eigenfunctions around the corner. Indeed, according to the extensive numerical examples in \cite{BLLW}, though the transmission eigenfunctions generically vanish around a corner, there are cases that the transmission eigenfunctions are not vanishing and instead they are localizing around a corner, especially when the corner is concave. Hence, it is mathematically intriguing and physically significant to thoroughly understand such a singularity formation of the transmission eigenfunctions and its connection to the corresponding vanishing behaviour. In \cite{ref1,ref3}, a regularity criterion of a different mathematical feature, but more physically related, has been investigated in connection to the vanishing property of the transmission eigenfunction. It is given in terms of the Hergoltz approximation of the transmission eigenfunction $v$ in \eqref{1.1}. The Herglotz approximation in a certain sense is the Fourier transform (in terms of the plane waves) of the eigenfunction $v$ who satisfies the homogeneous Helmholtz equation. Hence, the growth rate of the transformed function, i.e. the density function in the Herglotz wave, can naturally be used to characterize the regularity of the underlying wave function. This resembles the classical way of defining the Sobolev space via the Bessel potentials. In this paper, we shall explore along this direction and derive much sharper estimates to show that the vanishing property of the transmission eigenfunctions holds for a much broader class of functions in terms of the Herglotz approximation. The vanishing property of the transmission eigenfunctions derived in this paper include the corresponding results in \cite{ref1,ref3} as special cases.

\subsection{Statement of the main results and discussions}

In order to present a complete and comprehensive study, the statements of our main results are lengthy and technically involved. Nevertheless, in order to give the readers a global picture of our study, we briefly summarize the major findings in the following two theorems. To that end, we first introduce the Herglotz approximation.

For $g_{j} \in L^{2}\left(\mathbb{S}^{n-1}\right)$, we set
\begin{equation}
    v_{j}(\Bx)=\int_{\mathbb{S}^{n-1}} e^{i k \xi \cdot \Bx} g_{j}(\xi) \mathrm{d} \sigma(\xi), \quad \xi \in \mathbb{S}^{n-1}, \Bx \in \mathbb{R}^{n}.\label{2.3}
\end{equation}
$v_j$ is known as a Herglotz wave with kernel $g_j$. It is easy to see from (\ref{2.3}) that $v_j$ is formed by the superposition of plane waves and it is an entire solution to the Helmholtz equation $\Delta v_{j}+k^{2} v_{j}=0 $. Hence, $g_j$ can be regarded as the Fourier density of the wave function $v_j$ in terms of the plane waves. We have the following denseness property of the Hergoltz waves.
\begin{Lemma}[\cite{Wec}]\label{lem:Herg}
Let $\Omega \Subset \mathbb R^n$ be a bounded Lipschitz domain with a connected complement and  ${\mathscr H}_k$ be the space of all the Herglotz wave functions of the form \eqref{2.3}. Define
$$
{\mathscr S}_k(\Omega ) =  \{u\in C^\infty (\Omega);\ \Delta u+k^2u=0\}
$$
and
$$
{\mathscr H}_k(\Omega ) =  \{u|_\Omega;\  u\in {\mathscr H}_k\}.
$$
Then  ${\mathscr H}_k(\Omega )$ is dense in ${\mathscr S}_k(\Omega )  \cap  L^2 ( \Omega )$ with respect to the topology induced by the $H^1(\Omega)$-norm.
\end{Lemma}

\begin{thm}\label{thm:main1}
Consider the transmission eigenvalue problem \eqref{1.1} with $\eta \not \equiv0$. Let $\mathbf{x}_c\in\partial\Omega$ be a corner point in two and three dimensions and $\mathcal{N}_h$ be a neighbourhood of $\mathbf{x}_c$ within $\Omega$ with $h\in\mathbb{R}_+$ sufficiently small. Suppose that $(1+V) w$ and $\eta$ are both H\"older continuous on $\overline{\mathcal{N}_h}$ and $\partial\mathcal{N}_h\cap\partial\Omega$ respectively and $\eta(\mathbf{x}_c)\neq 0$. If there exist constants $C,\varrho$ and $\Upsilon$ with $C>0, \Upsilon>0 \text { and } \varrho<\Upsilon$ such that the transmission eigenfunction $v$ can be approximated in $H^{1}\left(\mathcal{N}_{h}\right)$ by the Herglotz functions $v_{j}, j=1,2, \ldots,$ with kernels $g_{j}$ satisfying
\begin{equation}
    \left\|v-v_{j}\right\|_{H^{1}\left(\mathcal{N}_{h}\right)} \leq j^{-\Upsilon}, \quad\left\|g_{j}\right\|_{L^{2}\left(\mathbb{S}^{n-1}\right)} \leq C j^{\varrho}, \label{eq:cond1}
\end{equation}
then $w$ and $v$ vanish near $\mathbf{x}_c$ in the sense of \eqref{eq:v1}.

More detailed results are respectively given in Theorems~\ref{Theorem 2.1} and \ref{Theorem 3.1} for the two and three dimensions.
\end{thm}

\begin{remark}\label{rem:a0}
As discussed earlier, the vanishing properties were investigated in \cite{ref1} under a similar setup to Theorem~\ref{thm:main1}. Compared to the corresponding results in \cite{ref1}, Theorem~\ref{thm:main1} has two significant improvements in the regularity requirements. First, the Herglotz approximation condition in \cite{ref1} was required to be
    \begin{equation}
        \left\|v-v_{j}\right\|_{H^{1}\left(\mathcal{N}_{h}\right)} \leq j^{-1-\Upsilon}, \quad\left\|g_{j}\right\|_{L^{2}\left(\mathbb{S}^{n-1}\right)} \leq C j^{\rho},\label{2.52}
    \end{equation}
    where the constants $C>0, \Upsilon>0 \text { and } 0<\varrho<1$. It is directly verified that the regularity condition (\ref{2.52}) is included in \eqref{eq:cond1} as a special case. Second, it was required in \cite{ref1} that $w-v$ is $H^2$-regular away from the corner point $\mathbf{x}_c$, and we remove this rather artificial regularity requirement in Theorem~\ref{thm:main1}.

\end{remark}

\begin{thm}\label{thm:main2}
Consider the transmission eigenvalue problem \eqref{1.1} with $\eta\equiv 0$. Under the same conditions as in Theorem \ref{thm:main1}, one has
    \begin{equation}
        \lim _{\rho \rightarrow+0} \frac{1}{m\left(B\left(\Bx_{c}, \rho\right)\cap\Omega\right)} \int_{B\left(\Bx_{c}, \rho\right)\cap\Omega} V(\Bx)w(\Bx) \mathrm{d} \Bx=0.\label{newre}
    \end{equation}
The similar result holds in the three-dimensional case with \eqref{eq:cond1} replaced to be \eqref{3.3}.

More detailed results are respectively given in Corollaries~\ref{maincor1} and \ref{maincor2} for the two and three dimensions.
\end{thm}

 \begin{remark}\label{rem:a1}
     If $V(\Bx)$ is continuous near the corner $\Bx_c$ and $V(\Bx_c)\neq0$, from the fact that
     \begin{equation*}
        \begin{split}
            &\lim _{\rho \rightarrow+0} \frac{1}{m\left(B\left(\mathbf{x}_{c}, \rho\right)\cap\Omega\right)} \int_{B\left(\mathbf{x}_{c}, \rho\right)\cap\Omega} V(\mathbf{x}) w(\mathbf{x}) \mathrm{d} \mathbf{x}\\
            =&V\left(\mathbf{x}_{c}\right) \lim _{\rho \rightarrow+0} \frac{1}{m\left(B\left(\mathbf{x}_{c}, \rho\right)\cap\Omega\right)} \int_{B\left(\mathbf{x}_{c}, \rho\right)\cap\Omega} w(\mathbf{x}) \mathrm{d} \mathbf{x},
        \end{split}
     \end{equation*}
     one can readily see that $w$ vanishes near $\mathbf{x}_c$, which in turn implies the vanishing of $v$ near $\mathbf{x}_c$ by noting that $w$ and $v$ possess the same traces on $\partial\Omega$.
 \end{remark}

\begin{remark}
The vanishing of the transmission eigenfunctions in the case $\eta\equiv 0$ was also studied in \cite{ref3, ref1}. The regularity requirement in \cite{ref1} is the same as that described in Remark~\ref{rem:a0}, whereas in \cite{ref3}, the Herglotz approximation was required to be
    \begin{equation}
        \left\|v-v_{j}\right\|_{L^{2}(\mathcal{N}_h)} \leq e^{-j}, \quad\left\|g_{j}\right\|_{L^{2}\left(\mathbb{S}^{n-1}\right)} \leq C(\ln j)^{\beta},
    \end{equation}
    where the constants $C>0 \text { and } 0<\beta<1 /(2 n+8),(n=2,3)$. It is directly verified that the corresponding results in \cite{ref3,ref1} are included into Theorems~\ref{thm:main1} and \ref{thm:main2} as special cases. Nevertheless, it is pointed out that in \cite{ref3}, the technical condition $(1+V)w$ being H\"older continuous is not required and instead it is required that $V$ is H\"older continuous.
\end{remark}

Finally, we would like to give two general remarks on the vanishing properties of the transmission eigenfunctions.

\begin{remark}\label{rem:g1}
The vanishing properties established in Theorems~\ref{thm:main1} and \ref{thm:main2} as well as those in \cite{ref3,ref1} are of a completely local feature. That is, all the results hold for the partial-data transmission eigenvalue problem, namely in \eqref{1.1} the transmission boundary conditions on $\partial\Omega$ is required to hold only in a small neighbourhood of the corner point. It is mentioned that a global rigidity result of the geometric structure of the transmission eigenfunctions was presented in \cite{CDHLW1}.
\end{remark}

\begin{remark}\label{rem:g2}
According to our earlier discussion, if the transmission eigenfunctions $w$ and $v$ are H\"older continuous around the corner, then both of them vanish near the corner. Hence, in order to search for the transmission eigenfunctions that are non-vanishing near corners, especially those numerically found in \cite{BLLW} which are actually locally localizing around corners, one should consider transmission eigenfunctions whose regularity lies between $H^1$ and $C^\alpha$, $\alpha\in (0, 1)$. By using properties of the Herglotz approximation (cf. \cite{ref4}), one can show (though not straightforward) that the regularity criterion \eqref{2.52} defines a set of functions which includes some functions that are less regular than $C^\alpha$, but also does not include some functions which are more regular than $C^\alpha$. Hence, the regularity characterization in terms of the Herglotz approximation is of a different feature from the standard Sobolev regularity. Nevertheless, Theorems~\ref{thm:main1} and \ref{thm:main2} indicate that the vanishing near corners is a generic local geometric property of the transmission eigenfunctions.
\end{remark}

In what follows, Sections 2 and 3 are respectively devoted to the vanishing properties of the transmission eigenfunctions in two and three dimensions.

\section{Vanishing properties in two dimensions}

To facilitate calculation and analysis, we introduce the two-dimensional polar coordinates $(r, \theta)$ such that $\Bx=\left(x_{1}, x_{2}\right)=(r \cos \theta, r \sin \theta) \in \mathbb{R}^{2}$. Set $B_{h}:=B_{h}(\mathbf{0})$ for $h\in\mathbb{R}_+$. Define an open sector in $\mathbb{R}^{2}$ with the boundary $\Gamma^{\pm}$ as follows,
\begin{equation}
    W=\left\{\Bx \in \mathbb{R}^{2} |\ \Bx \neq \mathbf{0}, \ \theta_{m}<\arg \left(x_{1}+\mathbf{i} x_{2}\right)<\theta_{M}\right\}.\label{2.1}
\end{equation}
where $-\pi<\theta_{m}<\theta_{M}<\pi, \mathbf{i}:=\sqrt{-1}$ and $\Gamma^{+}$ and $\Gamma^{-}$ respectively are
$\left(r, \theta_{M}\right)$ and $\left(r, \theta_{m}\right)$ with $r>0 .$ Define that
\begin{equation}
    S_{h}=W \cap B_{h}, \Gamma_{h}^{\pm}=\Gamma^{\pm} \cap B_{h}, \bar{S}_{h}=\overline{W} \cap B_{h}, \Lambda_{h}=S_{h} \cap \partial B_{h}, \text { and } \Sigma_{\Lambda_{h}}=S_{h} \backslash S_{h / 2}.\label{2.2}
\end{equation}

We shall make use a particular type of planar complex geometrical optics (CGO) solution, which was first introduced in \cite{ref2}.

\newtheorem{lemma}{Lemma}[section]
\begin{lemma}\label{Lemma 2.1}
    \cite[Lemma 2.2]{ref2} For $\Bx\in \RR^2$ denote $r=|\Bx|,\theta=arg(x_1+\mathbf{i}x_2).$ Define
\begin{equation}
    u_{0}(\Bx):=\exp \left(\sqrt{r}\left(\cos \left(\frac{\theta}{2}+\pi\right)+\mathbf{i} \sin \left(\frac{\theta}{2}+\pi\right)\right)\right),\label{2.4}
\end{equation}
then $\Delta u_0=0 \text { in } \RR^2 \setminus \RR_{0,-}^{2}$, where $\RR_{0,-}^{2}:=\left\{\Bx \in \RR^2|\ \Bx = (x_1,x_2);\ x_1<=0, x_2=0\right\}$ and $s \mapsto u_{0}(sx)$ decays exponentially in $R_{+} .$ Let $\alpha, s>0 .$ Then
\begin{equation}
    \int_{W}\left|u_{0}(s \Bx) \| \Bx\right|^{\alpha} \mathrm{d} \Bx \leq \frac{2\left(\theta_{M}-\theta_{m}\right) \Gamma(2 \alpha+4)}{\delta_{W}^{2 \alpha+4}} s^{-\alpha-2}.\label{2.5}
\end{equation}
\text {where } $\delta_{W}=-\max _{\theta_{m}<\theta<\theta_{M}} \cos (\theta / 2+\pi)>0 . $\text { Moreover }
\begin{equation}
    \int_{W} u_{0}(s \Bx) \mathrm{d} x=6 \mathbf{i}\left(e^{-2 \theta_{M} \mathbf{i}}-e^{-2 \theta_{m} \mathbf{i}}\right) s^{-2}.\label{2.6}
\end{equation}
and for $h>0$
\begin{equation}
    \int_{W \backslash B_{h}}\left|u_{0}(s \Bx)\right| \mathrm{d} \Bx \leq \frac{6\left(\theta_{M}-\theta_{m}\right)}{\delta_{W}^{4}} s^{-2} e^{-\delta_{W} \sqrt{h s} / 2}.\label{2.7}
\end{equation}
\end{lemma}

By direct calculations, one can obtain the following estimates for the CGO solution $u_0(s\Bx)$:
\begin{cor}\label{cor2.1}
   $u_0 \notin H^2(B_{\varepsilon})$ near the origin and $|u_0(s\Bx)|\leq1$ in $B_{\varepsilon}$ for sufficiently small $\varepsilon$.
\end{cor}
Furthermore, one has the following result:
\begin{cor}\label{cor2.2}
    The following estimates hold for the $L^2$ norm of $u_0$,
    \begin{align}
     \left\|u_{0}(s \Bx)\right\|_{L^{2}\left(S_{h}\right)}^{2}  & \leq \frac{\left(\theta_{M}-\theta_{m}\right) e^{-2 \sqrt{s \Theta} \delta_{W}} h^{2}}{2}, \quad \Theta\in (0, h),\\
       \left\|u_{0}(s \Bx)\right\|_{L^{2}\left(\Lambda_{h}\right)} & \leq \sqrt{h} e^{-\delta_{W} \sqrt{s h}} \sqrt{\theta_{M}-\theta_{m}},\\
       \left\|\partial_{\nu} u_{0}(s \Bx)\right\|_{L^{2}\left(\Lambda_{h}\right)} & \leq \frac{1}{2} \sqrt{s} e^{-\delta_{W} \sqrt{s h}} \sqrt{\theta_{M}-\theta_{m}},\\
        \left\|\partial_{\theta} u_{0}(s \Bx)\right\|_{L^{2}\left(\Lambda_{h}\right)} & \leq  \frac{\sqrt{s}}{2}h^2 e^{-\delta_{W} \sqrt{s h}} \sqrt{\theta_{M}-\theta_{m}},\\
 \left\||\Bx|^{\alpha} u_{0}(s \Bx)\right\|_{L^{2}\left(S_{h}\right)}^{2} & \leq s^{-(2 \alpha+2)} \frac{2\left(\theta_{M}-\theta_{m}\right)}{\left(4 \delta_{W}^{2}\right)^{2 \alpha+2}} \Gamma(4 \alpha+4),
    \end{align}
 where $\delta_{W}$ is defined in (\ref{2.5}) and $S_{h}$, $\Lambda_{h}$ are defined in (\ref{2.2}).
\end{cor}

\begin{proof}
    Using the integral mean value theorem, one can deduce that
    \begin{equation*}
        \begin{split}
            \left\|u_{0}(s \Bx)\right\|_{L^{2}\left(S_{h}\right)}^{2} &=\int_{0}^{h} r \mathrm{d} r \int_{\theta_{m}}^{\theta_{M}} e^{2 \sqrt{s r} \cos (\theta / 2+\pi)} \mathrm{d} \theta \\
            &\leq \int_{0}^{h} r \mathrm{d} r \int_{\theta_{m}}^{\theta_{M}} e^{-2 \sqrt{s r} \delta_{W}} \mathrm{d} \theta \\
        &=\frac{\left(\theta_{M}-\theta_{m}\right) e^{-2 \sqrt{s \Theta} \delta_{W}} h^{2}}{2},
        \end{split}
    \end{equation*}
where $\Theta\in (0, h)$.
On $\Lambda_{h}$, it can be seen that
    \begin{equation*}
        \begin{split}
            \left|u_{0}(s \Bx)\right| &=e^{\sqrt{s h} \cos (\theta / 2+\pi)} \leq e^{-\delta_{W} \sqrt{s h}} \\
        \left|\partial_{\nu} u_{0}(s \Bx)\right| &=\left|\frac{\sqrt{s} e^{i \cos (\theta / 2+\pi)}}{2 \sqrt{h}} e^{\sqrt{s h} \exp (i(\theta / 2+\pi))}\right| \leq \frac{1}{2} \sqrt{\frac{s}{h}} e^{-\delta_{W} \sqrt{s h}}\\
        \left|\partial_{\theta} u_{0}(s \Bx)\right| &=\left|\frac{\sqrt{sh} }{2 } e^{\sqrt{s h} \exp (i(\theta / 2+\pi))}\right| \leq \frac{\sqrt{sh} }{2 } e^{-\delta_{W} \sqrt{s h}},
        \end{split}
    \end{equation*}
all of which decay exponentially as $s \rightarrow \infty $. By straightforward calculations, one can show that
    \begin{equation*}
        \begin{split}
        \left\|u_{0}(s \Bx)\right\|_{L^{2}\left(\Lambda_{h}\right)} &\leq \sqrt{h}e^{-\delta_{W} \sqrt{s h}} \sqrt{\theta_{M}-\theta_{m}}, \\
        \left\|\partial_{\nu} u_{0}(s \Bx)\right\|_{L^{2}\left(\Lambda_{h}\right)} &\leq \frac{1}{2} \sqrt{s} e^{-\delta_{W} \sqrt{s h}} \sqrt{\theta_{M}-\theta_{m}},\\
       \left\|\partial_{\theta} u_{0}(s \Bx)\right\|_{L^{2}\left(\Lambda_{h}\right)} &\leq \frac{\sqrt{s}}{2}h^2 e^{-\delta_{W} \sqrt{s h}} \sqrt{\theta_{M}-\theta_{m}}.
    \end{split}
    \end{equation*}
Using polar coordinates, we can deduce that
\begin{equation}
    \begin{split}
        \quad\left\||\Bx|^{\alpha} u_{0}(s \Bx)\right\|_{L^{2}\left(S_{h}\right)}^{2}&=\int_{0}^{h} r \mathrm{d} r \int_{\theta_{m}}^{\theta_{M}} r^{2 \alpha} e^{2 \sqrt{s r} \cos (\theta / 2+\pi)} \mathrm{d} \theta \\
        &\leq \int_{0}^{h} r \mathrm{d} r \int_{\theta_{m}}^{\theta_{M}} r^{2 \alpha} e^{-2 \sqrt{s r} \delta w} \mathrm{d} \theta\\
        &=\left(\theta_{M}-\theta_{m}\right) \int_{0}^{h} r^{2 \alpha+1} e^{-2 \delta_{W} \sqrt{s r}} \mathrm{d} r \quad\left(t=2 \delta_{W} \sqrt{s r}\right) \\
        &=s^{-(2 \alpha+2)} \frac{2\left(\theta_{M}-\theta_{m}\right)}{\left(4 \delta_{W}^{2}\right)^{2 \alpha+2}} \int_{0}^{2 \delta_{W} \sqrt{s h}} t^{4 \alpha+3} e^{-t} \mathrm{d} r \\
        &\leq s^{-(2 \alpha+2)} \frac{2\left(\theta_{M}-\theta_{m}\right)}{\left(4 \delta_{W}^{2}\right)^{2 \alpha+2}}\Gamma(4 \alpha+4),
        \end{split}
\end{equation}
which completes the proof.
\end{proof}
Next, by direct computations and the compact embeddings of H\"older spaces, one can easily obtain the following result:
\begin{lemma}\label{lemma2.2}
    For the Herglotz wave function $v_j$ defined in \eqref{2.3} in two dimensions,
    \begin{equation}
        \begin{split}
            \left\|v_{j}\right\|_{C^{1}} &\leq \sqrt{2 \pi}(1+k)\left\|g_{j}\right\|_{L^{2}\left(\mathbb{S}^{n-1}\right)},\\
            \left\|v_{j}\right\|_{C^{\alpha}} &\leq \operatorname{diam}\left(S_{h}\right)^{1-\alpha}\left\|v_{j}\right\|_{C^{1}}.
        \end{split}
    \end{equation}
    where $0<\alpha<1$ and $\operatorname{diam}\left(S_{h}\right)$ is the diameter of $S_{h}$.
\end{lemma}

Furthermore, using the Jacobi-Anger expansion (\cite[Page 75]{ref4}) in $\RR^2$, we can obtain the following lemma.
\begin{lemma}\label{lemma2.3}
The Helgloltz wave function $v_j$ defined in \eqref{2.3} admits the following asymptotic expansion:
    \begin{equation}
        \begin{split}
            v_{j}(\Bx)&=J_{0}(k|\Bx|)\int_{\mathbb{S}^{n-1}} g_{j}(\theta) \mathrm{d} \sigma(\theta) + 2 \sum_{p=1}^{\infty} \mathbf{i}^{p} J_{p}(k|\Bx|) \int_{\mathbb{S}^{n-1}} g_{j}(\theta) \cos (p \varphi) \mathrm{d} \sigma(\theta)\\
        &:=v_{j}(0) J_{0}(k|\Bx|)+2 \sum_{p=1}^{\infty}  \mathbf{i}^{p} J_{p}(k|\Bx|) \gamma_{p j}, \quad \Bx \in \mathbb{R}^{2}.
        \end{split}
    \end{equation}
    where $J_{p}(t)$ is the p-th Bessel function of the first kind. Indeed,
    \begin{equation}
        J_{p}(t)=\frac{t^{p}}{2^{p} p !}+\frac{t^{p}}{2^{p}} \sum_{\ell=1}^{\infty} \frac{(-1)^{\ell} t^{2 \ell}}{4^{\ell}(\ell !)^{2}}, \quad  p=0, 1, 2, \ldots
    \end{equation}
\end{lemma}
\begin{lemma}\label{lemma2.4}
    Suppose $f(\Bx)\in C^{\alpha}$, then the following expansion holds near the origin:
    \begin{equation}
            f(\Bx)=f(\mathbf{0})+\delta f(\Bx),\quad |\delta f(\Bx)|\leq \left\|f\right\|_{C^\alpha} |\Bx|^{\alpha}.
    \end{equation}
\end{lemma}

 \begin{lemma}\label{lemma2.5}
     Let $v,w \in H^{1}(\Omega)$ be a pair of conductive transmission eigenfunctions to \eqref{1.1}, $D_{\varepsilon}=S_{h}\setminus B_{\varepsilon}$ for $0<\varepsilon<h$, $\eta \in C^{\alpha}\left(\bar{\Gamma}_{h}^{\pm}\right)$ for $0<\alpha<1$ and $\Gamma ^{\pm}_{h}$, $S_{h}$ be defined in \eqref{2.2}, then
     \begin{equation}
         \begin{split}
            \lim_{\varepsilon \rightarrow 0} \int _{D_{\varepsilon}} \Delta (v-w)u_{0}(s\Bx) \mathrm{d} \Bx
            = &\int_{\Lambda_{h}}\left(u_{0}(s \Bx) \partial_{\nu}(v-w)-(v-w) \partial_{\nu} u_{0}(s \Bx)\right) \mathrm{d} \sigma\\
            &-\int_{\Gamma_{h}^{\pm}} \eta(\Bx) u_{0}(s \Bx) v(\Bx) \mathrm{d} \sigma.
         \end{split}\label{lDe}
     \end{equation}
 \end{lemma}

 \begin{proof}
     Using Green's formula and the boundary condition in (\ref{1.1}), we can deduce that
     \begin{equation}
        \begin{split}
            &\int_{D_{\varepsilon}}\Delta (v-w) u_{0}(s \Bx) \mathrm{d} \Bx\\
            =&\int_{D_{\varepsilon}}(v-w) \Delta u_{0}(s \Bx) \mathrm{d} \Bx+\int_{\partial D_{\varepsilon}}\left(u_{0}(s \Bx) \partial_{\nu}(v-w)-(v-w) \partial_{\nu} u_{0}(s \Bx)\right) \mathrm{d} \sigma \\
            =&\int_{\Lambda_{h}}\left(u_{0}(s \Bx) \partial_{\nu}(v-w)-(v-w) \partial_{\nu} u_{0}(s \Bx)\right) \mathrm{d} \sigma \\
            &+\int_{\Lambda_{\varepsilon}}\left(u_{0}(s \Bx) \partial_{\nu}(v-w)-(v-w) \partial_{\nu} u_{0}(s \Bx)\right) \mathrm{d} \sigma\\
            &-\int_{\Gamma_{(\varepsilon,h)}^{\pm}} \eta(\Bx) u_{0}(s \Bx) v(\Bx) \mathrm{d} \sigma ,
        \end{split}\label{De}
    \end{equation}
where $\Gamma_{(\varepsilon, h)}^{\pm}=\Gamma^{\pm} \cap\left(B_{h} \backslash B_{\varepsilon}\right)$ and $\Lambda_{\varepsilon}=S_{h}\cap \partial B_{\varepsilon}$. Using the trace theorem, we have  $v \in L^{2}\left(\Gamma_{(0, \varepsilon)}^{\pm}\right)$, where $\Gamma_{(0, \varepsilon)}^{\pm}=\Gamma^{\pm} \cap B_{\varepsilon}$. By Corollary \ref{cor2.1}, one can see that
\begin{equation*}
    \lim _{\varepsilon \rightarrow 0} \int_{\Lambda_{\varepsilon}}\left(u_{0} \partial_{\nu}(v-w)-(v-w) \partial_{\nu} u_{0}\right) \mathrm{d} \sigma=0,
\end{equation*}
and
\begin{equation*}
    \lim _{\varepsilon \rightarrow 0} \int_{\Gamma_{(0, \varepsilon)}^{\pm}} \eta u_{0} v \mathrm{d} \sigma=0.
\end{equation*}
Thus (\ref{lDe}) is obtained by taking the limit on both sides of the equation (\ref{De}).
 \end{proof}

\begin{lemma}\label{lemma2.6}
    Suppose $\eta \in C^{\alpha}\left(\bar{\Gamma}_{h}^{\pm}\right)$ for $0<\alpha<1$, $\theta_{M},\theta_{m}$ are defined in \eqref{2.1} and $\theta_{M}-\theta_{m}\neq \pi$. Define
    \begin{equation}
        \omega(\theta)=-\cos (\theta / 2+\pi), \quad \mu(\theta)=-\cos (\theta / 2+\pi)-\mathbf{i} \sin (\theta / 2+\pi),\label{mu}
    \end{equation}
    and
    \begin{equation}
        I_{2}^{\pm}=\int_{\Gamma_{h}^{\pm}} \eta(\Bx) u_{0}(s \Bx) v_{j}(\Bx) \mathrm{d} \sigma.
    \end{equation}
    Then it holds that
    \begin{equation}
        \begin{split}
            I_{2}^{-}=& 2 \eta(\mathbf{0}) v_{j}(\mathbf{0}) s^{-1}\left(\mu\left(\theta_{m}\right)^{-2}-\mu\left(\theta_{m}\right)^{-2} e^{-\sqrt{s h} \mu\left(\theta_{m}\right)}-\mu\left(\theta_{m}\right)^{-1} \sqrt{s h} e^{-\sqrt{s h} \mu\left(\theta_{m}\right)}\right)\\
            &+v_{j}(\mathbf{0})\eta(\mathbf{0})I_{21}^{-} + \eta(\mathbf{0})I_{22}^{-}+I_{\eta}^{-},
        \end{split}\label{I2-}
    \end{equation}
    and
    \begin{equation}
        \begin{split}
            I_{2}^{+}=& 2 \eta(\mathbf{0}) v_{j}(\mathbf{0}) s^{-1}\left(\mu\left(\theta_{M}\right)^{-2}-\mu\left(\theta_{M}\right)^{-2} e^{-\sqrt{s h} \mu\left(\theta_{M}\right)}-\mu\left(\theta_{M}\right)^{-1} \sqrt{s h} e^{-\sqrt{s h} \mu\left(\theta_{M}\right)}\right)\\
            &+v_{j}(\mathbf{0})\eta(\mathbf{0})I_{21}^{+} + \eta(\mathbf{0})I_{22}^{+}+I_{\eta}^{+},
        \end{split}\label{I2+}
    \end{equation}
    where
    \begin{equation*}
        \begin{split}
            I_{21}^{-} &\leq  \mathcal{O}(s^{-3}), \quad I_{21}^{+} \leq  \mathcal{O}(s^{-3}),\\
            I_{22}^{-} &\leq \mathcal{O}(\|g_{j}\|_{L^{2}(\mathbb{S}^{n-1})}s^{-2}),\quad
            I_{22}^{+} \leq \mathcal{O}(\|g_{j}\|_{L^{2}(\mathbb{S}^{n-1})}s^{-2}),\\
            |I_{\eta}^{-}| &\leq \|\eta\|_{C^{\alpha}}\left(v_{j}(\mathbf{0})\mathcal{O}(s^{-1-\alpha}) + \mathcal{O}(\|g_{j}\|_{L^{2}(\mathbb{S}^{n-1})}s^{-2-\alpha})\right),\\
            |I_{\eta}^{+}| &\leq \|\eta\|_{C^{\alpha}}\left(v_{j}(\mathbf{0})\mathcal{O}(s^{-1-\alpha}) + \mathcal{O}(\|g_{j}\|_{L^{2}(\mathbb{S}^{n-1})}s^{-2-\alpha})\right).
        \end{split}
    \end{equation*}
\end{lemma}

\begin{proof}
    Using Lemma \ref{lemma2.4}, we have
    \begin{equation*}
        I_{2}^{-}=\eta(\mathbf{0}) \int_{\Gamma_{h}^{-}} u_{0}(s \Bx) v_{j}(\Bx) \mathrm{d} \sigma + \int_{\Gamma_{h}^{-}} \delta \eta(\Bx) u_{0}(s \Bx) v_{j}(\Bx) \mathrm{d} \sigma.
    \end{equation*}
  Note that $\omega(\theta)>0$ for $\theta_{m} \leq \theta \leq \theta_{M}$.
    By Lemma \ref{lemma2.3}, we can obtain that
    \begin{equation}
        \begin{split}
            \int_{\Gamma_{h}^{-}} u_{0}(s \Bx) v_{j}(\Bx) \mathrm{d} \sigma
            =&\int_{0}^{h}\left(v_{j}(\mathbf{0}) J_{0}(k r)+2 \sum_{p=1}^{\infty} \gamma_{p j} \mathbf{i}^{p} J_{p}(k r)\right) e^{-\sqrt{s r} \mu\left(\theta_{m}\right)} \mathrm{d} r \\
            =&v_{j}(\mathbf{0})\left[\int_{0}^{h} e^{-\sqrt{s r} \mu\left(\theta_{m}\right)} \mathrm{d} r+\sum_{p=1}^{\infty} \frac{(-1)^{p} k^{2 p}}{4^{p}(p !)^{2}} \int_{0}^{h} r^{2 p} e^{-\sqrt{s r} \mu\left(\theta_{m}\right)} \mathrm{d} r\right]\\
            &+\int_{0}^{h} 2 \sum_{p=1}^{\infty} \gamma_{p j} \mathbf{i}^{p} J_{p}(k r) e^{-\sqrt{s r} \mu\left(\theta_{m}\right)} \mathrm{d} r\\
            :=&v_{j}(\mathbf{0})\left[\int_{0}^{h} e^{-\sqrt{s r} \mu\left(\theta_{m}\right)} \mathrm{d} r + I_{21}^{-}\right] + I_{22}^{-}.
        \end{split}
    \end{equation}
    Furthermore, we have
\begin{equation}
    \int_{0}^{h} e^{-\sqrt{s r} \mu\left(\theta_{m}\right)} \mathrm{d} r = 2 s^{-1}\left(\mu\left(\theta_{m}\right)^{-2}-\mu\left(\theta_{m}\right)^{-2} e^{-\sqrt{s h} \mu\left(\theta_{m}\right)}-\mu\left(\theta_{m}\right)^{-1} \sqrt{s h} e^{-\sqrt{s h} \mu\left(\theta_{m}\right)}\right),
\end{equation}
and
\begin{equation}
    \left|I_{21}^{-} \right| \leq \sum_{p=1}^{\infty} \frac{k^{2 p} h^{2 p-2}}{4^{p}(p !)^{2}} \int_{0}^{h} r^{2} e^{-\sqrt{s r} \omega\left(\theta_{m}\right)} \mathrm{d} r=\mathcal{O}\left(s^{-3}\right).
\end{equation}
By using Lemma \ref{lemma2.3}, one can drive that
\begin{equation}
    \begin{split}
        \left|I_{22}^{-}\right|
        & \leq 2\left\|g_{j}\right\|_{L^{2}\left(\mathbb{S}^{n-1}\right)} \sum_{p=1}^{\infty} \left[\frac{k^{p}}{2^{p} p !} \int_{0}^{h} r^{p} e^{-\sqrt{s r} \omega\left(\theta_{m}\right)} \mathrm{d} r\right.
        \left.+\frac{(k h)^{p}}{2^{p}} \sum_{\ell=1}^{\infty} \frac{k^{2 \ell} h^{2(\ell-1)}}{4^{\ell}(\ell !)^{2}} \int_{0}^{h} r^{2} e^{-\sqrt{s r} \omega\left(\theta_{m}\right)} \mathrm{d} r\right] \\
        &\leq 2\left\|g_{j}\right\|_{L^{2}\left(\mathbb{S}^{n-1}\right)} \sum_{p=1}^{\infty}\left[\frac{k^{p} h^{p-1}}{2^{p} p !} \int_{0}^{h} r e^{-\sqrt{s r} \omega\left(\theta_{m}\right)} \mathrm{d} r+\mathcal{O}\left(s^{-3}\right)\right] \\
        &= \mathcal{O}\left(\left\|g_{j}\right\|_{L^{2}\left(\mathbb{S}^{n-1}\right)} s^{-2}\right),
    \end{split}
\end{equation}
where we assume that $kh<1$ for sufficiently small $h$.
Taking
\begin{equation}
    I_{\eta}^{-}=\int_{\Gamma_{h}^{-}} \delta \eta(\Bx) u_{0}(s \Bx) v_{j}(\Bx) \mathrm{d} \sigma \label{Ieta}.
\end{equation}
By Lemmas \ref{lemma2.3} and \ref{lemma2.4}, it holds that
\begin{equation}
    \begin{split}
        \left|I_{\eta}^{-}\right| \leq &\|\eta\|_{C^{\alpha}} \int_{0}^{h} r^{\alpha}\left(v_{j}(0) J_{0}(k r)+2 \sum_{p=1}^{\infty}\left|\gamma_{p j}\right| J_{p}(k r)\right) e^{-\sqrt{s r} \omega\left(\theta_{m}\right)} \mathrm{d} r\\
        = & \|\eta\|_{C^{\alpha}} v_{j}(0) (\int_{0}^{h} r^{\alpha} e^{-\sqrt{s r \omega}\left(\theta_{m}\right)} \mathrm{d} r + \sum_{p=1}^{\infty} \frac{(-1)^{p} k^{2 p}}{4^{p}(p !)^{2}} \int_{0}^{h} r^{\alpha+2 p} e^{-\sqrt{s r} \omega\left(\theta_{m}\right)} \mathrm{d} r)\\
        &+2 \|\eta\|_{C^{\alpha}} \sum_{p=1}^{\infty} \int_{0}^{h} r^{\alpha} \gamma_{p j} \mathbf{i}^{p} J_{p}(k r) e^{-\sqrt{s r} \mu\left(\theta_{m}\right)} \mathrm{d} r.
    \end{split}
\end{equation}
Since $\omega(\theta_{m})>0$, as $s\to \infty$, we have
\begin{equation}
    \int_{0}^{h} r^{\alpha} e^{-\sqrt{s r \omega}\left(\theta_{m}\right)} \mathrm{d} r = \mathcal{O}\left(s^{-1-{\alpha}}\right),
\end{equation}
and
\begin{equation}
    \left|\sum_{p=1}^{\infty} \frac{(-1)^{p} k^{2 p}}{4^{p}(p !)^{2}} \int_{0}^{h} r^{\alpha+2 p} e^{-\sqrt{s r} \omega\left(\theta_{m}\right)} \mathrm{d} r)\right| \leq \sum_{p=1}^{\infty} \frac{h^{2 p-2} k^{2 p}}{4^{p}(p !)^{2}} \int_{0}^{h} r^{\alpha+2} e^{-\sqrt{s r} \omega\left(\theta_{m}\right)} \mathrm{d} r=\mathcal{O}\left(s^{-3-\alpha}\right),
\end{equation}
for $kh<1$. Furthermore, by direct computations, one can obtain that
\begin{equation}
    \begin{split}
        &\left|\sum_{p=1}^{\infty} \int_{0}^{h} r^{\alpha} \gamma_{p j} \mathbf{i}^{p} J_{p}(k r) e^{-\sqrt{s r} \mu\left(\theta_{m}\right)} \mathrm{d} r\right|\\
        \leq & \left\|g_{j}\right\|_{L^{2}\left(\mathbb{S}^{n-1}\right)} \sum_{p=1}^{\infty} \left[\frac{k^{p}}{2^{p} p !} \int_{0}^{h} r^{p+\alpha} e^{-\sqrt{s r} \omega\left(\theta_{m}\right)} \mathrm{d} r\right.\\
        &\left.+\frac{k^{p}}{2^{p}} \sum_{\ell=1}^{\infty} \frac{k^{2 \ell} h^{2(\ell-1)}}{4^{\ell}(\ell !)^{2}}\left(\int_{0}^{h} r^{p+\alpha+2} e^{-\sqrt{s r} \omega\left(\theta_{m}\right)} \mathrm{d} r\right)\right] \\
        \leq & \left\|g_{j}\right\|_{L^{2}\left(\mathbb{S}^{n-1}\right)} \sum_{p=1}^{\infty} \left[\frac{k^{p}}{2^{p} p !} \int_{0}^{h} r^{p+\alpha} e^{-\sqrt{s r} \omega\left(\theta_{m}\right)} \mathrm{d} r\right.\\
        &\left.\left.+\frac{(k h)^{p}}{2^{p}} \sum_{\ell=1}^{\infty} \frac{k^{2 \ell} h^{2(\ell-1)}}{4^{\ell}(\ell !)^{2}}\left(\int_{0}^{h} r^{\alpha+2} e^{-\sqrt{s r} \omega\left(\theta_{m}\right)} \mathrm{d} r\right)\right]\right] \\
        \leq & \left\|g_{j}\right\|_{L^{2}\left(\mathbb{S}^{n-1}\right)} \sum_{p=1}^{\infty}\left[\frac{k^{p} h^{p-1}}{2^{p} p !} \int_{0}^{h} r^{\alpha+1} e^{-\sqrt{s r} \omega\left(\theta_{m}\right)} \mathrm{d} r+\mathcal{O}\left(s^{-\alpha-3}\right)\right]\\
        = & \mathcal{O}\left(\left\|g_{j}\right\|_{L^{2}\left(\mathbb{S}^{n-1}\right)} s^{-\alpha-2}\right).\label{Ieta3}
    \end{split}
\end{equation}
Combining (\ref{Ieta})-(\ref{Ieta3}), one finally obtains
\begin{equation}
    |I_{\eta}^{-}|\leq \left\| \eta \right\|_{C^{\alpha}}\left(v_j(\mathbf{0}) \mathcal{O}\left(s^{-1-\alpha}\right) + \mathcal{O}\left(\left\|g_{j}\right\|_{L^{2}\left(\mathbb{S}^{n-1}\right)} s^{-\alpha-2}\right)\right).\label{Ietan}
\end{equation}
Similarly, one can prove (\ref{I2+}), which completes the proof.
\end{proof}

\begin{lemma}\label{lemma2.7}
Suppose that $\eta \in C^{\alpha}\left(\bar{\Gamma}_{h}^{\pm}\right)$ for $0<\alpha<1$, $\Gamma ^{\pm}_{h}$ is defined in (\ref{2.2}), $\theta_{M},\theta_{m}$ are defined in (\ref{2.1}) and $\theta_{M}-\theta_{m}\neq \pi$. Define
    \begin{equation}
        \xi_{j}^{\pm}(s)=\int_{\Gamma_{h}^{\pm}} \eta(\Bx) u_{0}(s \Bx)\left(v(\Bx)-v_{j}(\Bx)\right) \mathrm{d} \sigma.
    \end{equation}
Then the following estimate holds,
    \begin{equation}
    \begin{split}
        \left|\xi_{j}^{\pm}(s)\right| \leq & C\Big(|\eta(0)| \frac{\sqrt{\theta_{M}-\theta_{m}} e^{-\sqrt{s \Theta \delta_{W}} h}}{\sqrt{2}}\\
        &+\|\eta\|_{C^{\alpha}}s^{-(\alpha+1)} \frac{\sqrt{2\left(\theta_{M}-\theta_{m}\right) \Gamma(4 \alpha+4)}}{\left(2 \delta_{W}\right)^{2 \alpha+2}}\Big) \left\|v-v_{j}\right\|_{H^{1}\left(S_{h}\right)},\label{xi}
        \end{split}
    \end{equation}
    where $\delta_{w}$ is defined in (\ref{2.5}).
\end{lemma}
\begin{proof}
    By Lemma \ref{lemma2.4}, the trace theorem and Cauchy-Schwarz inequality, one has
    \begin{equation*}
        \begin{split}
            \left|\xi_{j}^{\pm}(s)\right| & \leq |\eta(\mathbf{0})| \int_{\Gamma_{h}^{\pm}}\left|u_{0}(s \mathbf{x})\right| \left|v(\mathbf{x})-v_{j}(\mathbf{x})\right| \mathrm{d} \sigma+\|\eta\|_{C^{\alpha}} \int_{\Gamma_{h}^{\pm}} \left| \mathbf{x}\right|^{\alpha}\left|u_{0}(s \mathbf{x})\right| \left| v(\mathbf{x})-v_{j}(\mathbf{x})\right| \mathrm{d} \sigma\\
            & \leq|\eta(\mathbf{0})|\left\|v-v_{j}\right\|_{H^{1 / 2}\left(\Gamma_{h}^{\pm}\right)}\left\|u_{0}(s \mathbf{x})\right\|_{H^{-1 / 2}\left(\Gamma_{h}^{\pm}\right)} \\
            &+\|\eta\|_{C^{\alpha}}\left\|v-v_{j}\right\|_{H^{1 / 2}\left(\Gamma_{h}^{\pm}\right)}\left\|\left|\mathbf{x}\right|^{\alpha} u_{0}(s \mathbf{x})\right\|_{H^{-1 / 2}\left(\Gamma_{h}^{\pm}\right)} \\
            & \leq|\eta(\mathbf{0})|\left\|v-v_{j}\right\|_{H^{1}\left(S_{h}\right)}\left\|u_{0}(s \mathbf{x})\right\|_{L^{2}\left(S_{h}\right)}+\|\eta\|_{C^{\alpha}}\left\|v-v_{j}\right\|_{H^{1}\left(S_{h}\right)}\left\|\left.|\mathbf{x}\right|^{\alpha} u_{0}(s \mathbf{x})\right\|_{L^{2}\left(S_{h}\right)}.
            \end{split}
    \end{equation*}
    Combining with Corollary \ref{cor2.2}, one readily obtains (\ref{xi}).

    The proof is complete.
\end{proof}

We are in a position to present our first main result on the vanishing properties of the conductive transmission eigenfunctions $(v,w)$ in two dimensions.

\begin{thm}\label{Theorem 2.1}
    Let $v \in H^{1}(\Omega)$ and $w \in H^{1}(\Omega)$ be a pair of eigenfunctions to \eqref{1.1} associated with $k \in \mathbb{R}_{+} .$ Assume that the domain $\Omega \subset \mathbb{R}^{2}$ contains a corner $\Omega \cap W,$ where $\Bx_{c}$ is the vertex of $\Omega \cap W$ and $W$ is a sector defined in \eqref{2.1}. Moreover, there exits a sufficiently small neighbourhood $S_{h}$ (i.e. $h>0$ is sufficiently small) of $\Bx_{c}$ in $\Omega,$ where $S_{h}$ is defined in \eqref{2.2}, such that $q w \in C^{\alpha}\left(\bar{S}_{h}\right)$ with $q:=1+V$ and $\eta \in C^{\alpha}\left(\bar{\Gamma}_{h}^{\pm}\right)$ for $0<\alpha<1,$. If the following conditions are fulfilled:
\begin{enumerate}
\item[(a)] the transmission eigenfunction $v$ can be approximated in $H^{1}\left(S_{h}\right)$ by the Herglotz functions $v_{j}, j=1,2, \ldots,$ with kernels $g_{j}$ satisfying
\begin{equation}
    \left\|v-v_{j}\right\|_{H^{1}\left(S_{h}\right)} \leq j^{-\Upsilon}, \quad\left\|g_{j}\right\|_{L^{2}\left(\mathbb{S}^{n-1}\right)} \leq C j^{\varrho}.\label{2.8}
\end{equation}
for some constants $C,\varrho \text{ and }\Upsilon$ with $C>0, \Upsilon>0 \text { and } \varrho<\Upsilon$.
\item[(b)] the function $\eta(\Bx)$ doest not vanish at the corner, i.e.,
\begin{equation}
    \eta\left(\Bx_{c}\right) \neq 0. \label{2.9}
\end{equation}
\item[(c)] the angles $\theta_{m}$ and $\theta_{M}$ of the sector $W$ satisfy
\begin{equation}
    -\pi<\theta_{m}<\theta_{M}<\pi\ \text { and }\ \theta_{M}-\theta_{m} \neq \pi. \label{2.10}
\end{equation}
\end{enumerate}
Then it holds the following vanishing property:
\begin{equation}
    \lim _{\rho \rightarrow+0} \frac{1}{m\left(B\left(\Bx_{c}, \rho\right)\cap \Omega\right)} \int_{B\left(\Bx_{c}, \rho\right)\cap \Omega}|v(\Bx)| \mathrm{d} \Bx=0. \label{2.11}
\end{equation}
\end{thm}

\begin{proof}[Proof of Theorem \ref{Theorem 2.1}] It is known that $\Delta+k^{2}$ is invariant under rigid motions, we assume without loss of generality that $\Bx_{c}$ is the origin. From (\ref{1.1}), define
\begin{equation}
    \Delta v=-k^{2} v:=f_{1}, \quad \Delta w=-k^{2} q w:=f_{2}.\label{2.12}
\end{equation}
which in combination with the boundary conditions in (\ref{1.1}), yields that
\begin{equation*}
    \Delta(v-w)=f_{1}-f_{2} \text { in } S_{h}, \quad v-w=0, \quad \partial_{\nu}(v-w)=-\eta v \text { on } \Gamma_{h}^{\pm}.
\end{equation*}
For national simplicity, we define $f_{1j}(\Bx)=-k^2v_j$. Consider the following integral:
\begin{equation}
    \begin{split}
        \int_{S_h} \Delta (v-w)u_{0}(s\Bx)\mathrm{d} \Bx&=\int_{S_h} (k^{2}qw-k^{2}v)u_{0}(s\Bx)\mathrm{d} \Bx\\
        &=\int_{S_h} u_{0}(s\Bx)(f_{1j}-f_2)\mathrm{d} \Bx-k^2\int_{S_h}(v-v_j)u_{0}(s\Bx)\mathrm{d} \Bx\\
        &:=I_1+\Delta_{j}(s).
    \end{split}\label{Sh}
\end{equation}
By Corollary \ref{cor2.1}, one has that $u_0\notin H^2$ near the origin. Consider the domain $D_{\varepsilon}=S_h\backslash B_{\varepsilon}$ for $0<\varepsilon<h$. Using the fact that
\begin{equation}
    \int_{S_h} \Delta (v-w)u_{0}(s\Bx)\mathrm{d} \Bx=\lim _{\varepsilon \rightarrow 0} \int_{D_{\varepsilon}}\Delta (v-w) u_{0}(s \Bx) \mathrm{d} \Bx,
\end{equation}
and by Lemma \ref{lemma2.5}, one can drive that
\begin{equation}
    I_{1}+\Delta_{j}(s)=I_{3}-I_{2}^{\pm}-\xi_{j}^{\pm}(s),\label{Z}
\end{equation}
where $I_1$ and $\Delta_{j}(s)$ are defined in (\ref{Sh}), and
\begin{equation}
    \begin{split}
        I_{2}^{\pm}=&\int_{\Gamma_{h}^{\pm}} \eta(\Bx) u_{0}(s \Bx) v_{j}(\Bx) \mathrm{d} \sigma, \\
        I_{3}=&\int_{\Lambda_{h}}\left(u_{0}(s \Bx) \partial_{\nu}(v-w)-(v-w) \partial_{\nu} u_{0}(s \Bx)\right) \mathrm{d} \sigma,\\
        \xi_{j}^{\pm}(s)=&\int_{\Gamma_{h}^{\pm}} \eta(\Bx) u_{0}(s \Bx)\left(v(\Bx)-v_{j}(\Bx)\right) \mathrm{d} \sigma.\label{I23}
    \end{split}
\end{equation}
Recalling that $I_1$ is defined in (\ref{Sh}), by Lemma \ref{lemma2.4} and the compact embedding, one can deduce that
\begin{equation}
    I_{1}=\left({f}_{1 j}(0)-f_{2}(0)\right) \int_{S_{h}} u_{0}(s \Bx) \mathrm{d} \Bx+\int_{S_{h}} \delta f_{1 j}(\Bx) u_{0}(s \Bx) \mathrm{d} \Bx-\int_{S_{h}} \delta f_{2}(\Bx) u_{0}(s \Bx) \mathrm{d} \Bx.
    \label{I1}
\end{equation}
where $\delta f_{1j}(\Bx)$ and $\delta f_{2}(\Bx)$ are deduced by Lemma \ref{lemma2.4}.
Using the fact that
\begin{equation*}
    \int_{S_{h}} u_{0}(s \Bx) \mathrm{d} \Bx=\int_{W} u_{0}(s \Bx) \mathrm{d} \Bx-\int_{W \backslash S_{h}} u_{0}(s \Bx) \mathrm{d} \Bx,
\end{equation*}
and combining (\ref{2.6}) and (\ref{2.7}) in the Lemma \ref{Lemma 2.1}, it can be derived that
\begin{equation}
    \left|\int_{S_{h}} u_{0}\left(s \Bx\right) \mathrm{d} \Bx\right| \leq 6\left|e^{-2 \theta_{M} \mathbf{i}}-e^{-2 \theta_{m} \mathbf{i}}\right| s^{-2}+\frac{6\left(\theta_{M}-\theta_{m}\right)}{\delta_{W}^{4}} s^{-2} e^{-\frac{\delta_{W} \sqrt{h s}}{2}}.\label{U0}
\end{equation}
From Lemma \ref{Lemma 2.1} , Corollary \ref{cor2.2} , Lemmas \ref{lemma2.4} and \ref{lemma2.2}, one can deduce that
\begin{equation}
    \begin{split}
        \left|\int_{S_{h}} \delta {f}_{1 j}(\Bx) u_{0}(s \Bx) \mathrm{d} \Bx\right|  \leq &\int_{S_{h}}\left|\delta {f}_{1 j}(\Bx)\right|\left|u_{0}(s \Bx)\right| \mathrm{d} \Bx \\
        \leq &\left\|{f}_{1 j}\right\|_{C^{\alpha}} \int_{W}\left|u_{0}(s \Bx) \| \Bx\right|^{\alpha} \mathrm{d} \Bx \\
        \leq & \frac{2\left(\theta_{M}-\theta_{m}\right) \Gamma(2 \alpha+4)}{\delta_{W}^{2 \alpha+4}}\left\|{f}_{1 j}\right\|_{C^{\alpha}} s^{-\alpha-2}\\
        \leq & \frac{2 \sqrt{2 \pi}\left(\theta_{M}-\theta_{m}\right) \Gamma(2 \alpha+4)}{\delta_{W}^{2 \alpha+4}} k^{2} \operatorname{diam}\left(S_{h}\right)^{1-\alpha} \\
        &(1+k)Cj^{\varrho} s^{-\alpha-2},
    \end{split}\label{df1j}
\end{equation}
where $\delta f_{1j}(\Bx)$ is deduced by Lemma \ref{lemma2.4}.
Using the assumption (\ref{2.8}), we further otain that
\begin{equation}
        \left|\int_{S_{h}} \delta {f}_{1 j}(\Bx) u_{0}(s \Bx) \mathrm{d} \Bx\right|\leq  C' j^{\varrho} s^{-\alpha-2},\label{f1j}
\end{equation}
where $C'=(1+k)C \frac{2 \sqrt{2 \pi}\left(\theta_{M}-\theta_{m}\right) \Gamma(2 \alpha+4)}{\delta_{W}^{2 \alpha+4}} k^{2} \operatorname{diam}\left(S_{h}\right)^{1-\alpha} $ and C is defined in (\ref{2.8}).
Similarly, we have
\begin{equation}
    \left|\int_{S_{h}} \delta f_{2}(\Bx) u_{0}(s \Bx) \mathrm{d} \Bx\right| \leq \frac{2\left(\theta_{M}-\theta_{m}\right) \Gamma(2 \alpha+4)}{\delta_{W}^{2 \alpha+4}}\left\|f_{2}\right\|_{C^{\alpha}} s^{-\alpha-2}.\label{f2}
\end{equation}

As for the $\Delta _{j}(s)$ defined in (\ref{Sh}), using the Cauthy-Schwarz inequality, Corollary \ref{cor2.2} and the assumption (\ref{2.8}), we can drive that
\begin{equation}
    \begin{split}
        \left|\Delta_{j}(s)\right| &\leq k^{2}\left\|v-v_{j}\right\|_{L^{2}\left(S_{h}\right)}\left\|u_{0}(s \Bx)\right\|_{L^{2}\left(S_{h}\right)}\\
        & \leq \frac{\sqrt{\theta_{M}-\theta_{m}} k^{2} e^{-\sqrt{s \Theta} \delta_{w}} h}{\sqrt{2}} j^{-\Upsilon}.\label{D}
    \end{split}
\end{equation}

By using the H$\ddot{\mbox{o}}$lder inequality, Corollary \ref{cor2.2}, and the trace theorem, one can prove that
\begin{equation}
    \begin{split}
        \left|I_{3}\right| & \leq\left\|u_{0}(s \Bx)\right\|_{H^{1/2}\left(\Lambda_{h}\right)}\left\|\partial_{\nu}(v-w)\right\|_{H^{-1/2}\left(\Lambda_{h}\right)}+\left\|\partial_{\nu} u_{0}(s \Bx)\right\|_{L^{2}\left(\Lambda_{h}\right)}\|v-w\|_{L^{2}\left(\Lambda_{h}\right)} \\
        & \leq\left(\left\|u_{0}(s \Bx)\right\|_{H^1\left(\Lambda_{h}\right)}+\left\|\partial_{\nu} u_{0}(s \Bx)\right\|_{L^{2}\left(\Lambda_{h}\right)}\right)\|v-w\|_{H^{1}\left(\Sigma_{\Lambda_{h}}\right)} \leq C e^{-c^{\prime} \sqrt{s}},\label{I3}
    \end{split}
\end{equation}
where $c^{\prime}>0$ as $s \rightarrow \infty$.

By Lemma \ref{lemma2.6} and (\ref{I1}), multiplying s on the both sides of (\ref{Z}), we can deduce that
\begin{equation}
    \begin{split}
        2 v_{j}(\mathbf{0}) \eta(\mathbf{0})&\left[  \left(\mu\left(\theta_{M}\right)^{-2}-\mu\left(\theta_{M}\right)^{-2} e^{-\sqrt{s h}\left(\theta_{M}\right)}-\mu\left(\theta_{M}\right)^{-1} \sqrt{s h} e^{-\sqrt{s h} \mu\left(\theta_{M}\right)}\right)\right.\\
        &\left.+\left(\mu\left(\theta_{m}\right)^{-2}-\mu\left(\theta_{m}\right)^{-2} e^{-\sqrt{s h} \mu\left(\theta_{m}\right)}-\mu\left(\theta_{m}\right)^{-1} \sqrt{s h} e^{-\sqrt{s h} \mu\left(\theta_{m}\right)}\right)\right] \\
        =&s\left[I_{3}-\left({f}_{1 j}(\mathbf{0})-f_{2}(\mathbf{0})\right) \int_{S_{h}} u_{0}(s \Bx) \mathrm{d} \Bx-\Delta_{j}(s)\right.\\
        &\left.-v_{j}(\mathbf{0}) \eta(\mathbf{0})\left(I_{21}^{+}+I_{21}^{-}\right)-\eta(\mathbf{0})\left(I_{22}^{+}+I_{22}^{-}\right)\right.\\
        &\left.-I_{\eta}^{+}-I_{\eta}^{-}-\int_{S_{h}} \delta {f}_{1 j}(\Bx) u_{0}(s \Bx) \mathrm{d} \Bx+\int_{S_{h}} \delta f_{2}(\Bx) u_{0}(s \Bx) \mathrm{d} \Bx-\xi_{j}^{\pm}(s)\right].
    \end{split}
\end{equation}
Taking $s=j^{\beta}$, where $max\{\varrho,0\}<\beta<\Upsilon$, and letting $j\to \infty$, by (\ref{U0}), (\ref{f1j}), (\ref{f2}), (\ref{I3}), Lemmas \ref{lemma2.6} and \ref{lemma2.7}, we can deduce that
\begin{equation}
    \eta(\mathbf{0})\left(\mu\left(\theta_{m}\right)^{-2}+\mu\left(\theta_{M}\right)^{-2}\right) \lim _{j \rightarrow \infty} v_{j}(\mathbf{0})=0.\label{2.49}
\end{equation}
Using the assumption of (\ref{2.10}), one has that
\begin{equation}
    \mu\left(\theta_{m}\right)^{-2}+\mu\left(\theta_{M}\right)^{-2}=\frac{\left(\cos \theta_{m}+\cos \theta_{M}\right)+\mathbf{i}\left(\sin \theta_{m}+\sin \theta_{M}\right)}{\left(\cos \theta_{m}+\mathbf{i} \sin \theta_{m}\right)\left(\cos \theta_{M}+\mathbf{i} \sin \theta_{M}\right)}\neq 0.\label{2.50}
\end{equation}
By assumption (\ref{2.9}), it holds that $\eta(0)\neq 0$. Combining (\ref{2.49}) with (\ref{2.50}), we know that
\begin{equation*}
    \lim _{j \rightarrow \infty} v_{j}(\mathbf{0})=0.
\end{equation*}
Thus, it is an easy consequence that
\begin{equation*}
    \lim _{j \rightarrow \infty} |v_{j}(\mathbf{0})|=\lim _{\rho \rightarrow+0} \frac{1}{m\left(B\left(\Bx_{c}, \rho\right)\cap\Omega\right)} \int_{B\left(\Bx_{c}, \rho\right)\cap\Omega}\left|v_{j}(\Bx)\right| \mathrm{d} \Bx.
\end{equation*}
Combining with  the triangular inequality
\begin{equation*}
    \begin{split}
        &\lim _{\rho \rightarrow+0} \frac{1}{m\left(B\left(\Bx_{c}, \rho\right)\cap\Omega\right)} \int_{B\left(\Bx_{c}, \rho\right)\cap\Omega}|v(\Bx)| \mathrm{d} \Bx \\
        \leq & \lim _{j \rightarrow \infty}\left(\lim _{\rho \rightarrow+0} \frac{1}{m\left(B\left(\Bx_{c}, \rho\right)\cap\Omega\right)} \int_{B\left(\Bx_{c}, \rho\right)\cap\Omega}\left|v(\Bx)-v_{j}(\Bx)\right| \mathrm{d} \Bx\right.\\
        &\left.+\lim _{\rho \rightarrow+0} \frac{1}{m\left(B\left(\Bx_{c}, \rho\right)\cap\Omega\right)} \int_{B\left(\Bx_{c}, \rho\right)\cap\Omega}\left|v_{j}(\Bx)\right| \mathrm{d} \Bx\right),
        \end{split}
\end{equation*}
one finally arrives at
\begin{equation*}
    \lim _{\rho \rightarrow+0} \frac{1}{m\left(B\left(\Bx_{c}, \rho\right)\cap\Omega\right)} \int_{B\left(\Bx_{c}, \rho\right)\cap\Omega}|v(\Bx)| \mathrm{d} \Bx=0.
\end{equation*}

The proof is complete.
\end{proof}

We next consider the degenerate case of Theorem \ref{Theorem 2.1} with $\eta \equiv 0$. By slightly modifying our proof of Theorem \ref{Theorem 2.1}, we can show the following result.

\begin{cor}\label{maincor1}
    Let $v \in H^{1}(\Omega)$ and $w \in H^{1}(\Omega)$ be a pair of eigenfunctions to \eqref{1.1} with $\eta \equiv 0$ and $k \in \mathbb{R}_{+} $. Let $W$ and $S_h$ be the same as described in Theorem \ref{Theorem 2.1}. Assume that $q w \in C^{\alpha}\left(\bar{S}_{h}\right)$ for $0<\alpha<1$. Under the condition \eqref{2.10} and that the transmission eigenfunction $v$ can be approximated in $H^{1}(S_h)$ by the Herglotz functions $v_j$ with kernels $g_j$ satisfying
    \begin{equation}
        \left\|v-v_{j}\right\|_{H^{1}\left(S_{h}\right)} \leq j^{-\Upsilon}, \quad\left\|g_{j}\right\|_{L^{2}\left(\mathbb{S}^{n-1}\right)} \leq C j^{\varrho},\label{newas}
    \end{equation}
    for some constants $C,\varrho \text{ and }\Upsilon$ with $C>0, \Upsilon>0 \text { and } \varrho<\alpha \Upsilon/2$, one has
    \begin{equation}
        \lim _{\rho \rightarrow+0} \frac{1}{m\left(B\left(\Bx_{c}, \rho\right)\cap\Omega\right)} \int_{B\left(\Bx_{c}, \rho\right)\cap\Omega} V(\Bx)w(\Bx) \mathrm{d} \Bx=0.\label{newre}
    \end{equation}
\end{cor}
 \begin{proof}
     Since the proof is similar to that of Theorem \ref{Theorem 2.1}, we only outline some necessary modifications in the following. Without loss of generality, we assume that $\Bx_c=0$. Since $\eta (\Bx)\equiv 0$, from (\ref{Z}), (\ref{I23}) and (\ref{I1}), we have the following integral indentity
     \begin{equation}
        \left({f}_{1 j}(0)-f_{2}(0)\right) \int_{S_{h}} u_{0}(s \Bx) \mathrm{d} \Bx+\Delta_{j}(s) =I_3 - \int_{S_{h}} \delta f_{1 j}(\Bx) u_{0}(s \Bx) \mathrm{d} \Bx+\int_{S_{h}} \delta f_{2}(\Bx) u_{0}(s \Bx) \mathrm{d} \Bx,\label{newmain}
     \end{equation}
where $\Delta_{j}(s)$ and $I_{3}$ are defined in (\ref{Sh}) and (\ref{I23}), $\delta f_{1j}(\Bx)$ and $\delta f_{2}(\Bx)$ are deduced by Lemma \ref{lemma2.4} with $f_{2}(\Bx)$ defined in (\ref{2.12}) and $f_{1j}(\Bx)=-k^{2}v_{j}$. Multiplying $s^{2}$ on both sides of (\ref{newmain}), taking $s=J^{\beta}$ where $\max \{\varrho/\alpha,0\}<\beta<\Upsilon/2$, using the assumptions (\ref{newas}) and (\ref{2.10}), and by letting $j \to \infty$, from (\ref{U0}), (\ref{df1j}), (\ref{f2}), (\ref{D}) and (\ref{I3}), we can prove that
\begin{equation}
    \lim_{j\to \infty} v_{j}(\mathbf{0})= \frac{f_{2}(\mathbf{0})}{-k^2}.
\end{equation}
Since
\begin{equation*}
    \begin{split}
        &\lim _{j \rightarrow \infty} v_{j}(\mathbf{0})=\lim _{j \rightarrow \infty} \lim _{\rho \rightarrow+0} \frac{1}{m(B(\mathbf{0}, \rho)\cap\Omega)} \int_{B(\mathbf{0}, \rho)\cap\Omega} v_{j}(\mathbf{x}) \mathrm{d} \mathbf{x}\\
        =&\lim _{\rho \rightarrow+0} \frac{1}{m(B(\mathbf{0}, \rho)\cap\Omega)} \int_{B(\mathbf{0}, \rho)\cap\Omega} v(\mathbf{x}) \mathrm{d} \mathbf{x},
    \end{split}
\end{equation*}
and
\begin{equation*}
    \frac{f_{2}(\mathbf{0})}{-k^{2}}=\lim _{\rho \rightarrow+0} \frac{1}{m(B(\mathbf{0}, \rho)\cap\Omega)} \int_{B(\mathbf{0}, \rho)\cap\Omega} q w(\mathbf{x}) \mathrm{d} \mathbf{x},
\end{equation*}
which together with the fact that
\begin{equation*}
    \lim _{\rho \rightarrow+0} \frac{1}{m(B(\mathbf{0}, \rho)\cap\Omega)} \int_{B(\mathbf{0}, \rho)\cap\Omega} v(\mathbf{x}) \mathrm{d} \mathbf{x}=\lim _{\rho \rightarrow+0} \frac{1}{m(B(\mathbf{0}, \rho)\cap\Omega)} \int_{B(\mathbf{0}, \rho)\cap\Omega} w(\mathbf{x}) \mathrm{d} \mathbf{x},
\end{equation*}
readily implies \eqref{newre}.

The proof is complete.
 \end{proof}

\section{Vanishing properties in three dimensions}

In this section, we consider the vanishing properties of the transmission eigenfunctions in the three-dimensional case. We first introduce the (edge) corner geometry in the three-dimensional setting. It is described by $W \times(-M, M),$ where $W$ is a sector defined in (\ref{2.1}) and $M \in \mathbb{R}_{+} .$ It is readily seen that $W \times(-M, M)$ actually describes an edge singularity and we call it a $3 \mathrm{D}$ corner for notational unification. Suppose that the domain $\Omega \subset \mathbb{R}^{3}$ possesses a $3 \mathrm{D}$ corner. Let $\Bx_{c} \in \mathbb{R}^{2}$ be the vertex of $W$ and $x_{3}^{c} \in(-M, M) .$ Then $\left(\Bx_{c}, x_{3}^{c}\right)$ is defined as an edge point of $W \times(-M, M)$. In what follows, we use $\mathbf{x}=(\mathbf{x}', x_3)\in W\times (-M, M)$ to signify a point in the edge corner.
Consider the following conductive transmission eigenvalue problem:
\begin{equation}\label{3DE}
\begin{cases}
\big(\Delta+k^2(1+V)\big) w=0\ &\ \mbox{in}\ \ W\times (-M,M),\medskip\\
(\Delta+k^2) v=0\ &\ \mbox{in}\ \ W\times (-M, M),\medskip\\
w=v,\ \ \partial_\nu w=\partial_\nu v+\eta v\ &\ \mbox{on}\ \Gamma^\pm\times (-M, M),
\end{cases}
\end{equation}
where $\Gamma^{\pm}$ are the two boundary pieces of $W$.

For the subsequent use, we introduce the following dimension reduction operator.

\newtheorem{definition}{Definition}[section]
\begin{definition}\label{Definition 3.1}
    \cite[Definition 2.3.1]{ref1} Let $W \subset \mathbb{R}^{2}$ be defined in \eqref{2.1}, $M>0 .$ For a given function
$g$ with the domain $W \times(-M, M) .$ Pick up any point $x_{3}^{c} \in(-M, M) .$ Suppose $\psi \in C_{0}^{\infty}\left(\left(x_{3}^{c}-L, x_{3}^{c}+L\right)\right)$ is a nonnegative function and $\psi \neq 0,$ where $L$ is sufficiently small such that $\left(x_{3}^{c}-L, x_{3}^{c}+L\right) \subset(-M, M),$ and write $\Bx=\left(\Bx^{\prime},  x_{3}\right) \in \mathbb{R}^{3},  \Bx^{\prime} \in \mathbb{R}^{2}$.
The dimension reduction operator $\CR$ is defined by
\begin{equation}
    \CR(g)\left(\Bx^{\prime}\right)=\int_{x_{3}^{c}-L}^{x_{3}^{c}+L} \psi\left(x_{3}\right) g\left(\Bx^{\prime}, x_{3}\right) d x_{3}.\label{3.1}
\end{equation}
where $\Bx^{\prime} \in W$.
\end{definition}

The following lemma shows the regularity of the functions after applying the dimension reduction operator.
\begin{lemma}\label{Lemma 3.1}
     Let $g \in H^{1}(W \times(-M, M)) \cap C^{\alpha}(\bar{W} \times[-M, M]),$ where $0<\alpha<1$. Then
\begin{equation*}
    \CR(g)\left(\Bx^{\prime}\right) \in H^{1}(W) \cap C^{\alpha}(\bar{W}).
\end{equation*}
\end{lemma}

\begin{proof}
     We first show that $\CR:H^{1}(W\times(-M,M))\to H^{1}(W)$ is a bounded operator. Let $g\in H^{1}(W\times(-M,M))$, by the dominated convergence theorem, we know that $\partial_{\Bx^{\prime}}^{\alpha}\CR(g)(\Bx^{\prime})= \CR(\partial_{\Bx^{\prime}}^{\alpha}g)(\Bx^{\prime})$ for $\alpha=(i,j)$ with $ i,j=0,1$ and $i+j\leq 1$, so
    \begin{equation*}
        |\partial_{\Bx^{\prime}}^{\alpha}\CR(g)(\Bx^{\prime})| \leq \int_{x_{3}^{c}-L}^{x_{3}^{c}+L} \|\psi\|_{\infty}\left|\partial_{\Bx^{\prime}}^{\alpha}g(\Bx^{\prime},x_{3}) \right| dx_{3}.
    \end{equation*}
    Furthermore, by Minkowski integral, we have
    \begin{equation*}
        \|\CR(g)(\Bx^{\prime})\|_{H^{1}(W)} \leq \|\psi\|_{\infty} \|g(\Bx^{\prime},x_{3})\|_{H^{1}(W\times(-M,M))}.
    \end{equation*}
    When $g\in C^{\alpha}(\bar{W}\times[-M, M])$, it can be easily drived that
    \begin{equation*}
        |\CR(g)(\Bx^{\prime}) - \CR(g)(\mathbf{y}^{\prime})| \leq 2 \|\psi\|_{\infty} \|g\|_{C^{\alpha}}|\Bx^{\prime} - \mathbf{y}^{\prime}|^{\alpha}.
    \end{equation*}
    which means that $\CR(g)(\Bx^{\prime}) \in C^{\alpha}(\bar{W})$.
\end{proof}
Similar to Lemma \ref{lemma2.2}, we have the following lemma.
\begin{lemma}\label{Lemma 3.2}
    For the Herglotz wave function $v_j$ defined in \eqref{2.3} in three dimensions, it holds that
    \begin{equation}
        \begin{split}
            \|\CR(v_{j})\|_{C^{1}} & \leq 4L\sqrt{\pi}\|\psi\|_{C^{\infty}}(1+k)\|g_j\|_{L^{2}(\mathbb{S}^{2})},\\
            \|\CR(v_{j})\|_{C^{\alpha}} & \leq \operatorname{diam}(S_{h})^{1-\alpha}\|\CR(v_{j})\|_{C^{1}}.
        \end{split}
    \end{equation}
    where $0<\alpha<1$ and $\operatorname{diam}(S_{h})$ is the diameter of $S_h$.
\end{lemma}
Using the Jacobi-Anger expansion (\cite[Page 33]{ref4}) in $\RR^{3}$, one can derive the following result:
\begin{lemma}\label{Lemma 3.3}
The Herglotz function $v_j$ has the expansion in three dimensions as follows:
    \begin{equation}
        v_{j}(\Bx)=v_{j}(\mathbf{0}) j_{0}(k|\Bx|)+\sum_{\ell=1}^{\infty} \gamma_{\ell j} \mathbf{i}^{\ell}(2 \ell+1) j_{\ell}(k|\Bx|), \quad \Bx \in \mathbb{R}^{3},
    \end{equation}
    where
\begin{equation}\label{eq:nn1}
    v_{j}(\mathbf{0})=\int_{\mathbb{S}^{2}} g_{j}(\mathbf{d}) \mathrm{d} \sigma(\mathbf{d}), \quad \gamma_{\ell j}:=\int_{\mathbb{S}^{2}} g_{j}(\mathbf{d}) P_{\ell}(\cos (\varphi)) \mathrm{d} \sigma(\mathbf{d}), \quad \mathbf{d} \in \mathbb{S}^{2}.
\end{equation}
Here $j_{\ell}(t)$ is the $\ell$-th spherical Bessel function \cite{ref5}, and $\varphi$ is the angle between $\Bx$ and $\mathbf{d}$
Moreover, we have the explicit expression of $j_{\ell}(t)$ as
\begin{equation}
    j_{\ell}(t)=\frac{t^{\ell}}{(2 \ell+1) ! !}\left(1-\sum_{l=1}^{\infty} \frac{(-1)^{l} t^{2 l}}{2^{l} l ! N_{\ell, l}}\right),\ \ell=0,1,2,\dots,
\end{equation}
where $N_{\ell, l}=(2 \ell+3) \cdots(2 \ell+2 l+1)$.
\end{lemma}

Applying the dimension reduction operator to the above spherical Bessel function, we can obtain the following lemma.
\begin{lemma}\label{Lemma 3.4}
$\CR(j_0)(\Bx^{\prime})$ and $\CR\left(j_{\ell}\right)\left(\Bx^{\prime}\right)$ have the deformation as following:
    \begin{equation}
    \begin{split}
        \CR(j_0)(\Bx^{\prime})&=C(\psi)\left[1-\sum_{l=1}^{\infty} \frac{(-1)^{l} k^{2 l}}{2^{l} l !(2 l+1) ! !}\left(\left|\Bx^{\prime}\right|^{2}+a_{0, l}^{2}\right)^{l}\right],\\
        \CR\left(j_{\ell}\right)\left(\Bx^{\prime}\right)&=\frac{k^{\ell}\left(\left|\Bx^{\prime}\right|^{2}+a_{\ell}^{2}\right)^{(\ell-1) / 2}}{(2 \ell+1) ! !}\left[1-\sum_{l=1}^{\infty} \frac{(-1)^{l} k^{2 l}\left(\left|\Bx^{\prime}\right|^{2}+a_{\ell, l}^{2}\right)^{l}}{2^{l} l ! N_{\ell, l}}\right] C_{1}(\psi)\left|\Bx^{\prime}\right|^{2},
    \end{split}
   \end{equation}
   where
   $$N_{\ell, l}=(2 \ell+3) \cdots(2 \ell+2 l+1),\quad a_{0, l},a_{\ell},a_{\ell,l} \in[-L, L], $$
   and
   $$C(\psi)=\int_{-L}^{L} \psi\left(x_{3}\right) \mathrm{d} x_{3}, \quad  C_{1}(\psi)=\int_{-\arctan L /\left|\Bx^{\prime}\right|}^{\arctan L /\left|\Bx^{\prime}\right|} \psi\left(\left|\Bx^{\prime}\right| \tan \varpi\right) \sec ^{3} \varpi \mathrm{d} \varpi, $$
   with
   $$\varpi \in\left[-\arctan L /\left|\Bx^{\prime}\right|, \arctan L /\left|\Bx^{\prime}\right|\right].$$
\end{lemma}
\begin{proof}
    By the mean value theorem, we have that
    \begin{equation*}
        \begin{split}
            \CR\left(j_{0}\right)\left(\Bx^{\prime}\right) &=\int_{-L}^{L} \psi\left(x_{3}\right) j_{0}(k|\Bx|) \mathrm{d} x_{3} \\
        &=\int_{-L}^{L} \psi\left(x_{3}\right) \mathrm{d} x_{3}-\sum_{l=1}^{\infty} \frac{(-1)^{l} k^{2 l}}{2^{l} l !(2 l+1) ! !} \int_{-L}^{L} \psi\left(x_{3}\right)\left(\left|\Bx^{\prime}\right|^{2}+x_{3}^{2}\right)^{l} \mathrm{d} x_{3} \\
        &=C(\psi)\left[1-\sum_{l=1}^{\infty} \frac{(-1)^{l} k^{2 l}}{2^{l} l !(2 l+1) ! !}\left(\left|\Bx^{\prime}\right|^{2}+a_{0, l}^{2}\right)^{l}\right].
        \end{split}
    \end{equation*}
    where $C(\psi)=\int_{-L}^{L} \psi\left(x_{3}\right) \mathrm{d} x_{3}$ and $a_{0, l} \in[-L, L]$.
    For $\CR\left(j_{\ell}\right)\left(\Bx^{\prime}\right),$ using the integral mean value theorem, it can be drived that for $\ell=1,2, \ldots$
\begin{equation}
    \begin{split}
        &\int_{-L}^{L} \psi\left(x_{3}\right)\left(\left|\Bx^{\prime}\right|^{2}+x_{3}^{2}\right)^{\ell / 2} \mathrm{d} x_{3} \\
=&\left(\left|\Bx^{\prime}\right|^{2}+a_{\ell}^{2}\right)^{(\ell-1) / 2} \int_{-L}^{L} \psi\left(x_{3}\right)\left(\left|\Bx^{\prime}\right|^{2}+x_{3}^{2}\right)^{1 / 2} \mathrm{d} x_{3} \\
=&\left|\Bx^{\prime}\right|^{2}\left(\left|\Bx^{\prime}\right|^{2}+a_{\ell}^{2}\right)^{(\ell-1) / 2} \int_{-\arctan L /\left|\Bx^{\prime}\right|}^{\arctan L /\left|\Bx^{\prime}\right|} \psi\left(\left|\Bx^{\prime}\right| \tan \varpi\right) \sec ^{3} \varpi \mathrm{d} \varpi \\
:=&C_{1}(\psi)\left|\Bx^{\prime}\right|^{2}\left(\left|\Bx^{\prime}\right|^{2}+a_{\ell}^{2}\right)^{(\ell-1) / 2},
        \end{split}
\end{equation}
where $a_{\ell} \in[-L, L]$. If $L<\left|\Bx^{\prime}\right|$, we know that $0<\sec \varpi<\sqrt{\frac{L^{2}}{\left|\Bx^{\prime}\right|^{2}}+1}$, where $\varpi \in\left[-\arctan L /\left|\Bx^{\prime}\right|, \arctan L /\left|\Bx^{\prime}\right|\right]$. Hence, it can be derived that
\begin{equation}
    0<C_{1}(\psi)<2^{1 / 2} \pi\|\psi\|_{\infty}.
\end{equation}
Thus, for $l=1,2,\ldots$, one has that
\begin{equation*}
    \begin{split}
        \CR\left(j_{\ell}\right)\left(\Bx^{\prime}\right) &=\int_{-L}^{L} \psi\left(x_{3}\right) j_{\ell}(k|\Bx|)\, \mathrm{d} x_{3} \\
        &=\frac{k^{\ell}}{(2 \ell+1) ! !} \int_{-L}^{L} \psi\left(x_{3}\right)\left(\left|\Bx^{\prime}\right|^{2}+x_{3}^{2}\right)^{\ell / 2}\left(1-\sum_{l=1}^{\infty} \frac{(-1)^{l} k^{2 l}\left(\left|\Bx^{\prime}\right|^{2}+x_{3}^{2}\right)^{l}}{2^{l} l ! N_{\ell, l}}\right)\, \mathrm{d} x_{3} \\
        &=\frac{k^{\ell}\left(\left|\Bx^{\prime}\right|^{2}+a_{\ell}^{2}\right)^{(\ell-1) / 2}}{(2 \ell+1) ! !}\left[1-\sum_{l=1}^{\infty} \frac{(-1)^{l} k^{2 l}\left(\left|\Bx^{\prime}\right|^{2}+a_{\ell, l}^{2}\right)^{l}}{2^{l} l ! N_{\ell, l}}\right] C_{1}(\psi)\left|\Bx^{\prime}\right|^{2}.
        \end{split}
\end{equation*}
where $a_{\ell}, a_{\ell, l} \in[-L, L]$.
\end{proof}

We next derive several critical auxiliary lemmas.

\begin{lemma}\label{Lemma 3.5}
    Let $v,w\in H^{1}(W\times (-M,M))$ be a pair of conductive transmission eigenfunctions to \eqref{3DE} and $D_{\varepsilon}=S_{h}\setminus B_{\varepsilon}$ for $0<\varepsilon<h$, $\eta \in C^{\alpha}(\bar{\Gamma}^{\pm}_{h} \times [-M,M])$ for $0<\alpha<1$ and $\eta=\eta(\Bx^{\prime})$ is independent of $x_3$. Then it holds that
\begin{equation}
    \begin{split}
        &\lim_{\varepsilon \to \infty} \int_{D_{\varepsilon}}u_{0}(s\Bx^{\prime})\left(\Delta_{\Bx^{\prime}}\CR(v) - \Delta_{\Bx^{\prime}}\CR(w)  \right) \mathrm{d} \Bx^{\prime}\\
        =&\int_{\Lambda_{h}}\left(u_{0}\left(s \Bx^{\prime}\right) \partial_{\nu} \CR(v-w)\left(\Bx^{\prime}\right)-\CR(v-w)\left(\Bx^{\prime}\right) \partial_{\nu} u_{0}\left(s \Bx^{\prime}\right)\right) \mathrm{d} \sigma\\
        &-\int_{\Gamma_{ h}^{\pm}} \eta\left(\Bx^{\prime}\right) \CR(v)\left(\Bx^{\prime}\right) u_{0}\left(s \Bx^{\prime}\right) \mathrm{d} \sigma,
    \end{split}
\end{equation}
where $\Lambda_{h}=S_{h} \cap \partial B_{h}$ and $\Gamma_{ h}^{\pm}=\Gamma^{\pm}\cap B_{h}$.
\end{lemma}

\begin{proof}
    Since $w\left(\Bx^{\prime}, x_{3}\right)=v\left(\Bx^{\prime}, x_{3}\right)$ when $\Bx^{\prime} \in \Gamma$ and $-L<x_{3}<L$, we have
    \begin{equation}
        \CR(w)\left(\Bx^{\prime}\right)=\CR(v)\left(\Bx^{\prime}\right) \text { on } \Gamma.\label{B1}
    \end{equation}
    Similarly, using the fact that $\eta$ is independent of $x_3$, we can obtain that
    \begin{equation}
        \partial_{\nu} \CR(v)\left(\Bx^{\prime}\right)+\eta\left(\Bx^{\prime}\right) \CR(v)\left(\Bx^{\prime}\right)=\partial_{\nu} \CR(w)\left(\Bx^{\prime}\right) \text { on } \Gamma.\label{B2}
    \end{equation}
    Therefore, by Green's formula, we have
    \begin{equation*}
        \begin{split}
            &\int_{D_{\varepsilon}} \Delta_{\Bx^{\prime}}\left(\CR(v)\left(\Bx^{\prime}\right)-\CR(w)\left(\Bx^{\prime}\right)\right) u_{0}\left(s \Bx^{\prime}\right) \mathrm{d} \Bx^{\prime} \\
        &=\int_{\partial D_{\varepsilon}}\left(u_{0}\left(s \Bx^{\prime}\right) \partial_{\nu} \CR(v-w)\left(\Bx^{\prime}\right)-\CR(v-w)\left(\Bx^{\prime}\right) \partial_{\nu} u_{0}\left(s \Bx^{\prime}\right)\right) \mathrm{d} \sigma \\
        &=\int_{\Lambda_{h}}\left(u_{0}\left(s \Bx^{\prime}\right) \partial_{\nu} \CR(v-w)\left(\Bx^{\prime}\right)-\CR(v-w)\left(\Bx^{\prime}\right) \partial_{\nu} u_{0}\left(s \Bx^{\prime}\right)\right) \mathrm{d} \sigma \\
        &\quad+\int_{\Lambda_{\varepsilon}}\left(u_{0}\left(s \Bx^{\prime}\right) \partial_{\nu} \CR(v-w)\left(\Bx^{\prime}\right)-\CR(v-w)\left(\Bx^{\prime}\right) \partial_{\nu} u_{0}\left(s \Bx^{\prime}\right)\right) \mathrm{d} \sigma \\
        &\quad-\int_{\Gamma_{(\varepsilon, h)}^{\pm}} \eta\left(\Bx^{\prime}\right) \CR(v)\left(\Bx^{\prime}\right) u_{0}\left(s \Bx^{\prime}\right) \mathrm{d} \sigma,
        \end{split}
    \end{equation*}
    $\text { where } \Lambda_{h}=S_{h} \cap \partial B_{h}, \Lambda_{\varepsilon}=S_{h} \cap \partial B_{\varepsilon} \text { and } \Gamma_{(\varepsilon, h)}^{\pm}=\Gamma^{\pm} \cap\left(B_{h} \backslash B_{\varepsilon}\right)$.

Since $v,w \in H^{1}\left(S_{h} \times(-L, L)\right),$ from Lemma \ref{Lemma 3.1} we know that $\CR(v-w) \in H^1(S_h)$, and it can be derived that
\begin{equation*}
    \lim _{\varepsilon \rightarrow 0} \int_{\Lambda_{c}}\left(u_{0}\left(s \Bx^{\prime}\right) \partial_{\nu} \CR(v-w)\left(\Bx^{\prime}\right)-\CR(v-w)\left(\Bx^{\prime}\right) \partial_{\nu} u_{0}\left(s \Bx^{\prime}\right)\right) \mathrm{d} \sigma=0
\end{equation*}
Since $v \in H^{1}\left(\left(S_{h} \cap B_{\varepsilon}\right) \times(-L, L)\right)$, from Lemma \ref{Lemma 3.1}, we also know $\CR(v)\left(\Bx^{\prime}\right) \in H^{1}\left(S_{h} \cap B_{\varepsilon}\right)$. Using the trace theorem, we see that $\CR(v)\left(\Bx^{\prime}\right) \in L^{2}\left(\Gamma_{(0, e)}^{\pm}\right) $ where  $\Gamma_{(0, c)}^{\pm}=\Gamma^{\pm} \cap B_{\varepsilon}$. For sufficiently small $\varepsilon$, we further see that $\left|u_{0}\left(s \Bx^{\prime}\right)\right| \leq 1$ and $\eta \in C^{\alpha}\left(\bar{\Gamma}_{h}^{\pm} \times[-M, M]\right)$. Hence, we have
\begin{equation*}
    \lim _{\varepsilon \rightarrow 0} \int_{\Gamma_{(0, c)}^{\pm}} \eta\left(\Bx^{\prime}\right) \CR(v)\left(\Bx^{\prime}\right) u_{0}\left(s \Bx^{\prime}\right) \mathrm{d} \sigma=0.
\end{equation*}

The proof is complete.
\end{proof}

\begin{lemma}\label{Lemma 3.6}
    Suppose that $\eta \in C^{\alpha}\left(\bar{\Gamma}_{h}^{\pm} \times [-M,M]\right)$ for $0<\alpha<1$ and $\eta=\eta(\Bx^{\prime})$ is independent of $x_3$, and $\theta_{M},\theta_{m}$ are defined in \eqref{2.1} and $\theta_{M}-\theta_{m}\neq \pi$. Define
    \begin{equation}
        I_{2}^{\pm}=\int_{\Gamma_{h}^{\pm}} \eta\left(\Bx^{\prime}\right) u_{0}\left(s \Bx^{\prime}\right) \CR\left(v_{j}\right)\left(\Bx^{\prime}\right) \mathrm{d} \sigma.
    \end{equation}
    Then it hods that
    \begin{equation}
        \begin{split}
            I_{2}^{-}=&2 \eta(\mathbf{0}) v_{j}(\mathbf{0}) s^{-1}\left(\mu\left(\theta_{m}\right)^{-2}-\mu\left(\theta_{m}\right)^{-2} e^{-\sqrt{s h} \mu\left(\theta_{m}\right)}-\mu\left(\theta_{m}\right)^{-1} \sqrt{s h} e^{-\sqrt{s h} \mu\left(\theta_{m}\right)}\right) C_{2}^{-} \\
        &+v_{j}(\mathbf{0}) \eta(\mathbf{0}) I_{21}^{-}+\eta(\mathbf{0}) I_{22}^{-}+I_{\eta}^{-},
        \end{split}\label{3I2-}
    \end{equation}
    and
    \begin{equation}
        \begin{split}
            I_{2}^{+}=&2 \eta(\mathbf{0}) v_{j}(\mathbf{0}) s^{-1}\left(\mu\left(\theta_{M}\right)^{-2}-\mu\left(\theta_{M}\right)^{-2} e^{-\sqrt{s h} \mu\left(\theta_{M}\right)}-\mu\left(\theta_{M}\right)^{-1} \sqrt{s h} e^{-\sqrt{s h} \mu\left(\theta_{M}\right)}\right) C_{2}^{+} \\
            &+v_{j}(\mathbf{0}) \eta(\mathbf{0}) I_{21}^{+}+\eta(\mathbf{0}) I_{22}^{+}+I_{\eta}^{+}.
        \end{split}\label{3I2+}
    \end{equation}
    where $C_{2}^{\pm}$ are positive constants, and
\begin{equation*}
    \begin{split}
        I_{21}^{-} \leq & \mathcal{O}\left(s^{-3}\right), \quad I_{22}^{-} \leq \mathcal{O}\left(\left\|g_{j}\right\|_{L^{2}\left(\mathbb{S}^{2}\right)} s^{-3}\right),\\
        \left|I_{\eta}^{-}\right| \leq &\|\eta\|_{C^{\alpha}}\left(v_{j}(\mathbf{0}) \mathcal{O}\left(s^{-1-\alpha}\right)+\mathcal{O}\left(\left\|g_{j}\right\|_{L^{2}\left(\mathbb{S}^{2}\right)} s^{-3-\alpha}\right)\right),\\
        I_{21}^{+} \leq & \mathcal{O}\left(s^{-3}\right), \quad I_{22}^{+} \leq \mathcal{O}\left(\left\|g_{j}\right\|_{L^{2}\left(\mathbb{S}^{2}\right)} s^{-3}\right),\\
        \left|I_{\eta}^{+}\right| \leq &\|\eta\|_{C^{\alpha}}\left(v_{j}(\mathbf{0}) \mathcal{O}\left(s^{-1-\alpha}\right)+\mathcal{O}\left(\left\|g_{j}\right\|_{L^{2}\left(\mathbb{S}^{2}\right)} s^{-3-\alpha}\right)\right).
    \end{split}
\end{equation*}
\end{lemma}

\begin{proof}
    Using Lemma \ref{lemma2.4}, we have
    \begin{equation*}
        I_{2}^{-}=\eta(\mathbf{0}) \int_{\Gamma_{h}^{-}} u_{0}\left(s \Bx^{\prime}\right) \CR\left(v_{j}\right)\left(\Bx^{\prime}\right) \mathrm{d} \sigma + \int_{\Gamma_{h}^{-}} \delta \eta\left(\Bx^{\prime}\right) u_{0}\left(s \Bx^{\prime}\right) \CR\left(v_{j}\right)\left(\Bx^{\prime}\right) \mathrm{d} \sigma.
    \end{equation*}
    Recall the definition in (\ref{mu}). Combining Lemmas \ref{Lemma 3.3} and \ref{Lemma 3.4}, we can obtain that
    \begin{equation*}
        \begin{split}
            &\int_{\Gamma_{h}^{-}} u_{0}\left(s \Bx^{\prime}\right) \CR\left(v_{j}\right)\left(\Bx^{\prime}\right) \mathrm{d} \sigma\\
            =&v_{j}(\mathbf{0}) \int_{\Gamma_{h}^{-}} u_{0}\left(s \Bx^{\prime}\right) \CR\left(j_{0}\right)\left(\Bx^{\prime}\right) \mathrm{d} \sigma+\sum_{\ell=1}^{\infty} \gamma_{\ell j} \mathbf{i}^{\ell}(2 \ell+1) \int_{\Gamma_{h}^{-}} u_{0}\left(s \Bx^{\prime}\right) \CR\left(j_{\ell}\right)\left(\Bx^{\prime}\right) \mathrm{d} \sigma.
        \end{split}
    \end{equation*}
    Using Newton's Binomial expansion (see also Lemma \ref{Lemma 3.4}), we have
    \begin{equation}
        \begin{split}
            &\int_{\Gamma_{h}^{-}} u_{0}\left(s \Bx^{\prime}\right) \CR\left(j_{0}\right)\left(\Bx^{\prime}\right) \mathrm{d} \sigma\\
            =&C(\psi) \int_{0}^{h}\left[1-\sum_{l=1}^{\infty} \frac{(-1)^{l} k^{2 l}}{(2 l+1) ! !}\left(r^{2}+a_{0, l}^{2}\right)^{l}\right] e^{-\sqrt{s r} \mu\left(\theta_{m}\right)} \mathrm{d} r\\
            =& C(\psi)\left[1-\sum_{l=1}^{\infty} \frac{(-1)^{l} k^{2 l}}{(2 l+1) ! !} a_{0, l}^{2 l}\right] \int_{0}^{h} e^{-\sqrt{s r} \mu\left(\theta_{m}\right)} \mathrm{d} r \\
            &-C(\psi) \sum_{l=1}^{\infty} \frac{(-1)^{l} k^{2 l}}{(2 l+1) ! !}\left(\sum_{i_{1}=1}^{l} C\left(l, i_{1}\right) a_{0, l}^{2\left(l-i_{1}\right)} \int_{0}^{h} r^{2 i_{1}} e^{-\sqrt{s r} \mu\left(\theta_{m}\right)} \mathrm{d} r\right)\\
            :=&2 s^{-1}\left(\mu\left(\theta_{m}\right)^{-2}-\mu\left(\theta_{m}\right)^{-2} e^{-\sqrt{s h} \mu\left(\theta_{m}\right)}-\mu\left(\theta_{m}\right)^{-1} \sqrt{s h} e^{-\sqrt{s h} \mu\left(\theta_{m}\right)}\right)C_{2}^{-} + I_{21}^{-},
        \end{split}\label{a}
    \end{equation}
    where $C(\psi)=\int_{-L}^{L} \psi\left(x_{3}\right) \mathrm{d} x_{3}>0$, $C\left(l, i_{1}\right)=\frac{l !}{i_{1} !\left(l-i_{1}\right) !}$ is the combinatorial number of order $l$ and $C_{2}^{-}=C(\psi)\left[1-\sum_{l=1}^{\infty} \frac{(-1)^{l} k^{2 l}}{(2 l+1) ! !} a_{0, l}^{2 l}\right]$. By choosing $L$ such that $kL<1$, we have that
    \begin{equation*}
        \left|\sum_{l=1}^{\infty} \frac{(-1)^{l} k^{2 l}}{(2 l+1) ! !} a_{0, l}^{2 l}\right| \leq \sum_{l=1}^{\infty}(k L)^{2 l}=\frac{(k L)^{2}}{1-(k L)^{2}}.
    \end{equation*}
    Therefore, we can deduce that
    \begin{equation}
        0<\frac{C(\psi)\left(1-2(k L)^{2}\right)}{1-(k L)^{2}} \leq C_{2}^{-} \leq \frac{C(\psi)}{1-(k L)^{2}}.\label{C2-}
    \end{equation}
    For $I_{21}^{-}$, choosing $h$ and $L$ such that $k^{2}(h^{2}+L^{2})<1$, we can deduce that
    \begin{equation}
        \begin{split}
            |I_{21}^{-}| & \leq|C(\psi)| \sum_{l=1}^{\infty} \frac{k^{2 l}}{(2 l+1) ! !} \sum_{i_{1}=1}^{l} C\left(l, i_{1}\right) h^{2\left(i_{1}-1\right)} L^{2\left(l-i_{1}\right)} \int_{0}^{h} r^{2} e^{-\sqrt{s r} \omega\left(\theta_{m}\right)} \mathrm{d} r \\
            &=|C(\psi)| \sum_{l=1}^{\infty} \frac{k^{2 l}}{(2 l+1) ! ! h^{2}} \sum_{i_{1}=1}^{l} C\left(l, i_{1}\right) h^{2 i_{1}} L^{2\left(l-i_{1}\right)} \int_{0}^{h} r^{2} e^{-\sqrt{s r} \omega\left(\theta_{m}\right)} \mathrm{d} r \\
            &=|C(\psi)| \sum_{l=1}^{\infty} \frac{k^{2 l}}{(2 l+1) ! ! h^{2}}\left(\left(h^{2}+L^{2}\right)^{l}-L^{2 l}\right) \int_{0}^{h} r^{2} e^{-\sqrt{s r} \omega\left(\theta_{m}\right)} \mathrm{d} r \\
            & \leq 2 L\|\psi\|_{\infty} \sum_{l=1}^{\infty} \frac{k^{2 l}}{(2 l+1) ! ! h^{2}}\left(\left(h^{2}+L^{2}\right)^{l}-L^{2 l}\right) \mathcal{O}\left(s^{-3}\right) \\
            &=\mathcal{O}\left(s^{-3}\right).
        \end{split}\label{b}
    \end{equation}
    Taking
    \begin{equation*}
        I_{22}^{-}=\sum_{\ell=1}^{\infty} \gamma_{\ell j} \mathbf{i}^{\ell}(2 \ell+1) \int_{\Gamma_{h}^{-}} u_{0}\left(s \Bx^{\prime}\right) \CR\left(j_{\ell}\right)\left(\Bx^{\prime}\right) \mathrm{d} \sigma,
    \end{equation*}
we then have
    \begin{equation}
        \begin{split}
            \left|I_{22}^{-}\right| \leq & C_{1}(\psi)\left\|g_{j}\right\|_{L^{2}\left(\mathbb{S}^{2}\right)} \\
        & \cdot \sum_{\ell=1}^{\infty} \int_{0}^{h} r^{2} e^{-\sqrt{s r \omega}\left(\theta_{m}\right)} \frac{k^{\ell}\left(|r|^{2}+a_{\ell}^{2}\right)^{(\ell-1) / 2}}{(2 \ell-1) ! !}\left|1-\sum_{l=1}^{\infty} \frac{(-1)^{l} k^{2 l}\left(|r|^{2}+a_{\ell, l}^{2}\right)^{l}}{2^{l} l ! N_{\ell, l}}\right| \mathrm{d} r \\
        =& C_{1}(\psi)\left\|g_{j}\right\|_{L^{2}\left(\mathbb{S}^{2}\right)} \sum_{\ell=1}^{\infty} \frac{k^{\ell}\left(\left|\beta_{\ell}\right|^{2}+a_{\ell}^{2}\right)^{(\ell-1) / 2}}{(2 \ell-1) ! !}\left|1-\sum_{l=1}^{\infty} \frac{(-1)^{l} k^{2 l}\left(\left|\beta_{\ell, l}\right|^{2}+a_{\ell, l}^{2}\right)^{l}}{2^{l} l ! N_{\ell, l}}\right| \\
        & \cdot \int_{0}^{h} r^{2} e^{-\sqrt{s r} \omega\left(\theta_{m}\right)} \mathrm{d} r \\
        =& \mathcal{O}\left(\left\|g_{j}\right\|_{L^{2}\left(\mathbb{S}^{2}\right)} s^{-3}\right),
        \end{split}\label{c}
    \end{equation}
    where $C_1(\psi)$ is defined in Lemma \ref{Lemma 3.4}.
    Taking
    \begin{equation*}
        I_{\eta}^{-}=\int_{\Gamma_{h}^{-}} \delta \eta\left(\Bx^{\prime}\right) u_{0}\left(s \Bx^{\prime}\right) \CR\left(v_{j}\right)\left(\Bx^{\prime}\right) \mathrm{d} \sigma,
    \end{equation*}
   and by Lemmas \ref{Lemma 3.3} and \ref{Lemma 3.4}, we can obtain that
    \begin{equation*}
        \begin{split}
            I_{\eta}^-&=v_j(0)\int_{\Gamma_{h}^{-}} \delta \eta\left(\Bx^{\prime}\right) u_{0}\left(s \Bx^{\prime}\right) \CR\left(j_{0}\right)\left(\Bx^{\prime}\right) \mathrm{d} \sigma \\
        &\quad+ \sum_{\ell=1}^{\infty} \gamma_{\ell j} \mathbf{i}^{\ell}(2 \ell+1) \int_{\Gamma_{h}^{-}} \delta \eta\left(\Bx^{\prime}\right) u_{0}\left(s \Bx^{\prime}\right) \CR\left(j_{\ell}\right)\left(\Bx^{\prime}\right) \mathrm{d} \sigma.
        \end{split}
    \end{equation*}
    Furthermore, we have
    \begin{equation}
        \begin{split}
            &\left|\int_{\Gamma_{h}^{-}} \delta \eta\left(\Bx^{\prime}\right) u_{0}\left(s \Bx^{\prime}\right) \CR\left(j_{0}\right)\left(\Bx^{\prime}\right) \mathrm{d} \sigma\right|\\
             & \leq|C(\psi)|\|\eta\|_{C^{\alpha}} \int_{0}^{h} r^{\alpha}\left|1-\sum_{l=1}^{\infty} \frac{(-1)^{l} k^{2 l}}{(2 l+1) ! !}\left(r^{2}+a_{0, l}^{2}\right)^{l}\right| e^{-\sqrt{s r} \omega\left(\theta_{m}\right)} \mathrm{d} r \\
            &=2 L\|\psi\|_{\infty}\|\eta\|_{C^{\alpha}}\left|1-\sum_{l=1}^{\infty} \frac{(-1)^{l} k^{2 l}}{(2 l+1) ! !}\left(\beta_{0, l}^{2}+a_{0, l}^{2}\right)^{l}\right| \int_{0}^{h} r^{\alpha} e^{-\sqrt{s r} \omega\left(\theta_{m}\right)} \mathrm{d} r\\
            &\leq \mathcal{O}\left(s^{-\alpha-1}\right),
            \end{split}\label{d}
    \end{equation}
    where $\beta_{0, l} \in[0, h]$ such that $k^{2}\left(\beta_{0, l}^{2}+a_{0, l}^{2}\right) \leq k^{2}\left(h^{2}+L^{2}\right)<1$ for sufficiently small $h$ and $L$. Next, we can deduce that
    \begin{equation}
        \begin{split}
            &\left|\sum_{\ell=1}^{\infty} \gamma_{\ell j} \mathbf{i}^{\ell}(2 \ell+1) \int_{\Gamma_{h}^{-}} \delta \eta\left(\Bx^{\prime}\right) u_{0}\left(s \Bx^{\prime}\right) \CR\left(j_{\ell}\right)\left(\Bx^{\prime}\right) \mathrm{d} \sigma\right|\\
            \leq & C_{1}(\psi)\|\eta\|_{C^{\alpha}} \\
            & \cdot \sum_{\ell=1}^{\infty}\left|\gamma_{\ell j}\right| \int_{0}^{h} r^{\alpha} \frac{k^{\ell}\left(r^{2}+a_{\ell}^{2}\right)^{(\ell-1) / 2}}{(2 \ell-1) ! !}\left|1-\sum_{l=1}^{\infty} \frac{k^{2 l}\left(r^{2}+a_{\ell, l}^{2}\right)^{l}}{2^{l} l ! N_{\ell, l}}\right| r^{2} e^{-\sqrt{s r} \omega\left(\theta_{m}\right)} \mathrm{d} r \\
            \leq & 2C_{1}(\psi)\|\eta\|_{C^{\alpha}}\left\|g_{j}\right\|_{L^{2}\left(\mathbb{S}^{2}\right)} \int_{0}^{h} r^{2+\alpha} e^{-\sqrt{s r} \omega\left(\theta_{m}\right)} \mathrm{d} r \\
            & \cdot \sum_{\ell=1}^{\infty} \frac{k^{\ell}\left(\beta_{\ell}^{2}+a_{\ell}^{2}\right)^{(\ell-1) / 2}}{(2 \ell-1) ! !}\left|1-\sum_{l=1}^{\infty} \frac{k^{2 l}\left(\beta_{\ell, l}^{2}+a_{\ell, l}^{2}\right)^{l}}{2^{l} l ! N_{\ell, l}}\right| \\
            =& \mathcal{O}\left(\left\|g_{j}\right\|_{L^{2}\left(\mathbb{S}^{2}\right)} s^{-\alpha-3}\right),
        \end{split}\label{e}
    \end{equation}
where we have used the estimates $k^{2}\left(\beta_{\ell}^{2}+a_{\ell}^{2}\right) \leq k^{2}\left(h^{2}+L^{2}\right)<1$, $k^{2}\left(\beta_{\ell, l}^{2}+a_{\ell, l}^{2}\right) \leq$ $k^{2}\left(h^{2}+L^{2}\right)<1$ for sufficiently small $h$ and $L,$ as well as the estimate
\begin{equation*}
    \left|\gamma_{\ell j}\right|=\Big|\int_{\mathbb{S}^{2}} g_{j}(\mathbf{d}) P_{\ell}(\hat{\Bx}|_{\Gamma_h^-}\cdot\mathbf{d}) \mathrm{d} \sigma(\mathbf{d})\Big| \leq \sqrt{2\pi}\sqrt{\frac{2}{2l+1}}\left\|g_{j}\right\|_{L^{2}\left(\mathbb{S}^{2}\right)}\leq 2\sqrt{\pi}\left\|g_{j}\right\|_{L^{2}\left(\mathbb{S}^{2}\right)}.
\end{equation*}

Finally, by combining (\ref{a}), (\ref{b}), (\ref{c}), (\ref{d}) and (\ref{e}), we can prove (\ref{3I2-}). Using a similar argument, one can prove (\ref{3I2+}).

The proof is complete.
\end{proof}

\begin{lemma}\label{lemma3.7}
    Let $\eta \in C^{\alpha}\left(\bar{\Gamma}_{h}^{\pm} \times [-M,M]\right)$ for $0<\alpha<1$ and $\eta=\eta(\Bx^{\prime})$ be independent of $x_3$, $\theta_{M},\theta_{m}$ be defined in (\ref{2.1}) and $\theta_{M}-\theta_{m}\neq \pi$. Set
    \begin{equation}
        \xi_{j}^{\pm}(s) =\int_{\Gamma_{h}^{\pm}} \eta\left(\Bx^{\prime}\right) u_{0}\left(s \Bx^{\prime}\right) \CR\left(v\left(\Bx^{\prime}, x_3\right)-v_{j}\left(\Bx^{\prime}, x_3\right)\right) \mathrm{d} \sigma.
    \end{equation}
    Then it holds that
    \begin{align*}
        & \left|\xi_{j}^{\pm}(s)\right|  \\
        \leq & C\|\psi\|_{\infty}\left(|\eta(\mathbf{0})|\left\|u_{0}\left(s \Bx^{\prime}\right)\right\|_{L^{2}\left(S_{h}\right)}+\|\eta\|_{C^{\alpha}}\left\|\left|\Bx^{\prime}\right|^{\alpha} u_{0}\left(s \Bx^{\prime}\right)\right\|_{L^{2}\left(S_{h}\right)}\right) \left\|v-v_{j}\right\|_{H^{1}\left(S_{h} \times(-L, L)\right)},
    \end{align*}
    where $C$ is a positive constant.
\end{lemma}
\begin{proof}
    Using Cauchy-Schwarz inequality and the trace theorem, we can deduce as follows
\begin{equation*}
    \begin{split}
    &    \left|\xi_{j}^{\pm}(s)\right| \leq |\eta(\mathbf{0})| \int_{\Gamma_{h}^{\pm}}\left|u_{0}\left(s \Bx^{\prime}\right) \| \CR\left(v\left(\Bx^{\prime}, x_3\right)-v_{j}\left(\Bx^{\prime}, x_3\right)\right)\right| \mathrm{d} \sigma \\
        &+\|\eta\|_{C^{\alpha}} \int_{\Gamma_{h}^{\pm}}\left|\Bx^{\prime}\right|^{\alpha}\left|u_{0}\left(s \Bx^{\prime}\right) \| \CR\left(v\left(\Bx^{\prime}, x_3\right)-v_{j}\left(\Bx^{\prime}, x_3\right)\right)\right| \mathrm{d} \sigma \\
        \leq &|\eta(\mathbf{0})|\left\|\CR\left(v-v_{j}\right)\right\|_{H^{1 / 2}\left(\Gamma_{h}^{\pm}\right)}\left\|u_{0}\left(s \Bx^{\prime}\right)\right\|_{H^{-1 / 2}\left(\Gamma_{h}^{\pm}\right)} \\
        &+\left.\|\eta\|_{C^{\alpha}}\left\|\CR\left(v-v_{j}\right)\right\|_{H^{1 / 2}\left(\Gamma_{h}^{\pm}\right)}\|\| \Bx^{\prime}\right|^{\alpha} u_{0}\left(s \Bx^{\prime}\right) \|_{H^{-1 / 2}\left(\Gamma_{h}^{\pm}\right)} \\
        \leq & | \eta(\mathbf{0})\left\|\CR\left(v-v_{j}\right)\right\|_{H^{1}\left(S_{h}\right)}\left\|u_{0}\left(s \Bx^{\prime}\right)\right\|_{L^{2}\left(S_{h}\right)} \\
        &+\left.\|\eta\|_{C^{\alpha}}\left\|\CR\left(v-v_{j}\right)\right\|_{H^{1}\left(S_{h}\right)}\|\| \Bx^{\prime}\right|^{\alpha} u_{0}\left(s \Bx^{\prime}\right) \|_{L^{2}\left(S_{h}\right)} \\
        \leq & C\|\psi\|_{\infty}\left\|v-v_{j}\right\|_{H^{1}\left(S_{h} \times(-L, L)\right)}\left(|\eta(\mathbf{0})|\left\|u_{0}\left(s \Bx^{\prime}\right)\right\|_{L^{2}\left(S_{h}\right)}+\|\eta\|_{C^{\alpha}}\left\|\left|\Bx^{\prime}\right|^{\alpha} u_{0}\left(s \Bx^{\prime}\right)\right\|_{L^{2}\left(S_{h}\right)}\right),
        \end{split}
\end{equation*}
which readily completes the proof.
\end{proof}

We are now in a position of to the vanishing properties of the conductive transmission eigenfunctions $(v,w)$ in the three-dimensional case, and we have the following theorem.

\begin{thm}\label{Theorem 3.1}
    Let $v, w \in H^{1}(W \times(-M, M))$ be a pair of eigenfunctions to \eqref{3DE} associated with $k\in \RR_{+}$, where $W \subset \mathbb{R}^{2}$ is defined in \eqref{2.1} and $M>0.$ For any fixed $x_{3}^{c} \in(-M, M)$ and $L>0$ defined in Definition \ref{Definition 3.1}, we suppose that $L$ is sufficiently small such that $\left(x_{3}^{c}-L, x_{3}^{c}+L\right) \subset(-M, M)$. Moreover, there exists a sufficiently small neighbourhood $S_{h}$ of $\Bx_{c} \in \mathbb{R}^{2}$ such that $q w \in C^{\alpha}\left(\bar{S}_{h} \times[-M, M]\right)$ and $\eta \in C^{\alpha}\left(\bar{\Gamma}_{h}^{\pm} \times[-M, M]\right)$ for $0<\alpha<1$, where $q:=1+V$. If the following conditions are fulfilled:
 \begin{enumerate}
\item[(a)] the transmission eigenfunction $v$ can be approximated in $H^{1}\left(S_{h} \times(-M, M)\right) b y$
the Herglotz functions $v_{j}, j=1,2, \ldots,$ with kernels $g_{j}$ satisfying
\begin{equation}
    \left\|v-v_{j}\right\|_{H^{1}\left(S_{h} \times(-M, M)\right)} \leq j^{-\Upsilon}, \quad\left\|g_{j}\right\|_{L^{2}\left(\mathbb{S}^{2}\right)} \leq C j^{\varrho},\label{3.3}
\end{equation}
for some constants $C,\varrho \text{ and }\Upsilon$ with $C>0, \Upsilon>0 \text { and } \varrho<(1+\alpha)\Upsilon$;
\item[(b)] the function $\eta=\eta\left(\Bx^{\prime}\right)$ is independent of $x_{3}$ and
\begin{equation}
    \eta\left(\Bx_{c}\right) \neq 0; \label{3.4}
\end{equation}
\item[(c)] the angles $\theta_{m}$ and $\theta_{M}$ of the sector $W$ satisfy
\begin{equation}
    -\pi<\theta_{m}<\theta_{M}<\pi \text { and } \theta_{M}-\theta_{m} \neq \pi; \label{3.5}
\end{equation}
\end{enumerate}
then for every edge point $\left(\Bx_{c}, x_3^{c}\right) \in \mathbb{R}^{3}$ of $W \times(-M, M)$ where $x_3^{c} \in(-M, M),$ one has
\[
\lim _{\rho \rightarrow+0} \frac{1}{m\left(B\left(\left(\Bx_{c}, x_3^{c}\right), \rho\right)\cap (W\times (-M,M))\right)} \int_{B\left(\left(\Bx_{c}, x_3\right), \rho\right)\cap (W\times (-M,M))}|v(\Bx)| \mathrm{d} \Bx=0.
\]
\end{thm}


\begin{proof}
 For an edge point $\left(\Bx_{c}, x_{3}^{c}\right) \in W \times(-M, M),$ without loss of generality, we assume that the vertex $\Bx_{c}$ of the sector $W \subset \mathbb{R}^{2}$ is located at the origin of $\mathbb{R}^{2}$ and $x_{3}^{c}=0$. By direct calculations, we have
 \begin{equation}
     \begin{split}
        \Delta_{\Bx^{\prime}} \CR(v)\left(\Bx^{\prime}\right)&=\Delta_{\Bx^{\prime}} \int_{-L}^{L} \psi(x_3)v(\Bx^{\prime},x_n) \mathrm{d}x_3\\
        &=\int_{-L}^{L} \psi(x_3) \left(-k^{2}v(\Bx^{\prime},x_n)-\partial_{x_3}^{2}v(\Bx^{\prime},x_n)\right) \mathrm{d}x_3\\
        &=-\int_{-L}^{L} \psi(x_3)\partial_{x_3}^{2}v(\Bx^{\prime},x_n)\mathrm{d}x_3 - k^{2}\CR(v)(\Bx^{\prime})\\
        &=\int_{-L}^{L} \psi^{\prime \prime}\left(x_3\right) v\left(\Bx^{\prime}, x_3\right) d x_3-k^{2} \CR(v)\left(\Bx^{\prime}\right).
     \end{split}
 \end{equation}
 Similarly, we can obtain that
 \begin{equation}
    \Delta_{\Bx^{\prime}} \CR(w)\left(\Bx^{\prime}\right) =\int_{-L}^{L} \psi^{\prime \prime}\left(x_3\right) w\left(\Bx^{\prime}, x_3\right) d x_3-k^{2} \CR(q w)\left(\Bx^{\prime}\right).
 \end{equation}
Therefore, we have
\begin{equation}
    \begin{split}
        &\Delta_{\Bx^{\prime}} \CR(v)\left(\Bx^{\prime}\right) - \Delta_{\Bx^{\prime}} \CR(w)\left(\Bx^{\prime}\right) \\
        =& \int_{-L}^{L} \psi^{\prime \prime}\left(x_3\right)\left(v\left(\Bx^{\prime}, x_3\right)-w\left(\Bx^{\prime}, x_3\right)\right) \mathrm{d} x_3+k^{2} \CR(q w)\left(\Bx^{\prime}\right)-k^{2} \CR(v)\left(\Bx^{\prime}\right)\\
        :=& F_{1}\left(\Bx^{\prime}\right)+ F_{2}\left(\Bx^{\prime}\right)+ F_{3}\left(\Bx^{\prime}\right).
    \end{split}\label{F123}
\end{equation}
Next, we set
\begin{equation}
    F_{3 j}\left(\Bx^{\prime}\right)=-k^{2} \CR\left(v_{j}\right)\left(\Bx^{\prime}\right), \label{F3j}
\end{equation}
and consider the following integral
\begin{equation}
    \begin{split}
        &\int_{S_{h}}\left(\Delta_{\Bx^{\prime}} \CR(v)\left(\Bx^{\prime}\right) - \Delta_{\Bx^{\prime}} \CR(w)\left(\Bx^{\prime}\right)\right)u_{0}\left(s \Bx^{\prime}\right)\mathrm{d} \Bx^{\prime} \\
        = &\int_{S_{h}}\left(F_{1}\left(\Bx^{\prime}\right)+F_{2}\left(\Bx^{\prime}\right)+F_{3 j}\left(\Bx^{\prime}\right)\right) u_{0}\left(s \Bx^{\prime}\right) \mathrm{d} \Bx^{\prime} + \int_{S_{h}}\left(F_{3}\left(\Bx^{\prime}\right)-F_{3 j}\left(\Bx^{\prime}\right)\right)\mathrm{d} \Bx^{\prime}\\
        :=& I_{1} + \Delta_{j}(s).
    \end{split}\label{3Djs}
\end{equation}
Using the fact that
\begin{equation}
\begin{split}
    &\int_{S_{h}}\left(\Delta_{\Bx^{\prime}} \CR(v)\left(\Bx^{\prime}\right) - \Delta_{\Bx^{\prime}} \CR(w)\left(\Bx^{\prime}\right)\right)u_{0}\left(s \Bx^{\prime}\right)\mathrm{d} \Bx^{\prime}\\
    =& \lim _{\varepsilon \rightarrow 0}\int_{D_{\varepsilon}} \left(\Delta_{\Bx^{\prime}} \CR(v)\left(\Bx^{\prime}\right) - \Delta_{\Bx^{\prime}} \CR(w)\left(\Bx^{\prime}\right)\right)u_{0}\left(s \Bx^{\prime}\right)\mathrm{d} \Bx^{\prime},
\end{split}
\end{equation}
where $D_{\varepsilon}=S_{h} \backslash B_{\varepsilon} \text { for } 0<\varepsilon<h$, and by Lemma \ref{Lemma 3.5}, we can deduce that
\begin{equation}
    I_{1}+\Delta_{j}(s)=I_{3}-I_{2}^{\pm}-\xi_{j}^{\pm}(s), \label{3Z}
\end{equation}
where $\xi_{j}^{\pm}(s)$ is defined in Lemma \ref{lemma3.7}, $I_{2}^{\pm}$ is defined in Lemma \ref{Lemma 3.6}, and
\begin{equation}
    I_{3}=\int_{\Lambda_{h}}\left(u_{0}\left(s \Bx^{\prime}\right) \partial_{\nu} \CR(v-w)-\CR(v-w) \partial_{\nu} u_{0}\left(s \Bx^{\prime}\right)\right) \mathrm{d} \sigma,\label{e3I3}
\end{equation}
with $\Lambda_{h}=S_{h}\cap \partial B_{h}$.
Since $v-w\in H^1(S_h \times (-L,L))$ and $qw \in C^\alpha (\bar{S}_h \times [-L,L]), \alpha \in (0,1)$, from Lemma \ref{Lemma 3.1} we know that $F_{1}\left(\Bx^{\prime}\right) \in C^{\alpha}(\bar{S}_h)$ and $F_{2}\left(\Bx^{\prime}\right) \in C^\alpha (\bar{S}_h)$. In addition, we have $\CR(v_{j})(\Bx^{\prime}) \in C^\alpha (\bar{S}_h)$. Therefore, by Lemma \ref{lemma2.4}, we have
\begin{equation}
    \begin{split}
        I_{1}=&\left(F_{1}(\mathbf{0})+F_{2}(\mathbf{0})+F_{3 j}(\mathbf{0})\right) \int_{S_{h}} u_{0}\left(s \Bx^{\prime}\right) \mathrm{d} \Bx^{\prime}+\int_{S_{h}} \delta F_{1}\left(\Bx^{\prime}\right) u_{0}\left(s \Bx^{\prime}\right) \mathrm{d} \Bx^{\prime} \\
        &+\int_{S_{h}} \delta F_{2}\left(\Bx^{\prime}\right) u_{0}\left(s \Bx^{\prime}\right) \mathrm{d} \Bx^{\prime}+\int_{S_{h}} \delta F_{3 j}\left(\Bx^{\prime}\right) u_{0}\left(s \Bx^{\prime}\right) \mathrm{d} \Bx^{\prime}.
        \end{split}\label{3I1}
\end{equation}
where $\delta F_{1}(\Bx')$, $\delta F_{2}(\Bx')$ and $\delta F_{3j}(\Bx')$ are deduced by Lemma \ref{lemma2.4} with $F_{1}(\Bx')$, $F_{2}(\Bx')$ and $F_{3j}(\Bx')$ defined in (\ref{F123}) and (\ref{F3j}).
By Lemmas \ref{Lemma 3.2} and \ref{Lemma 2.1}, we can deduce that
\begin{equation}
    \begin{split}
        \left|\int_{S_{h}} \delta F_{1}\left(\Bx^{\prime}\right) u_{0}\left(s \Bx^{\prime}\right) \mathrm{d} \Bx^{\prime}\right| \leq & \frac{2\left\|F_{1}\right\|_{C^{\alpha}}\left(\theta_{M}-\theta_{m}\right) \Gamma(2 \alpha+4)}{\delta_{W}^{2 \alpha+4}} s^{-\alpha-2}, \\
        \left|\int_{S_{h}} \delta F_{2}\left(\Bx^{\prime}\right) u_{0}\left(s \Bx^{\prime}\right) \mathrm{d} \Bx^{\prime}\right| \leq & \frac{2\left\|F_{2}\right\|_{C^{\alpha}}\left(\theta_{M}-\theta_{m}\right) \Gamma(2 \alpha+4)}{\delta_{W}^{2 \alpha+4}} s^{-\alpha-2},\\
        \left|\int_{S_{h}} \delta F_{3 j}\left(\Bx^{\prime}\right) u_{0}\left(s \Bx^{\prime}\right) \mathrm{d} \Bx^{\prime}\right| \leq & \frac{8 L \sqrt{\pi}\|\psi\|_{C^{\infty}}\left(\theta_{M}-\theta_{m}\right) \Gamma(2 \alpha+4)}{\delta_{W}^{2 \alpha+4}} k^{2} \operatorname{diam}\left(S_{h}\right)^{1-\alpha} \\
        & \times (1+k)\left\|g_{j}\right\|_{L^{2}\left(\mathbb{S}^{n-1}\right)} s^{-\alpha-2}.
    \end{split}\label{I1N}
\end{equation}

For $\Delta_{j}(s)$, using Cauthy-Schwarz inequality, Corollary \ref{cor2.2} and the assumption (\ref{3.3}), we can drive that
\begin{equation}
    \begin{split}
        \left|\Delta_{j}(s)\right| & \leq k^{2}\left\|\CR(v)-\CR\left(v_{j}\right)\right\|_{L^{2}\left(S_{h}\right)}\left\|u_{0}\left(s \Bx^{\prime}\right)\right\|_{L^{2}\left(S_{h}\right)} \\
        & \leq \frac{k^{2}\|\psi\|_{\infty} \sqrt{C(L, h)\left(\theta_{M}-\theta_{m}\right)} e^{-\sqrt{s \Theta} \delta_{W}} h}{\sqrt{2}} j^{-\Upsilon},
    \end{split}\label{3DS}
\end{equation}
where $C(L, h)$ is a positive constant depending on $L$ and $h$ and $\Theta \in[0, h]$.

By Lemma 3.1, and the same arguments in (\ref{I3}), we have
\begin{equation}
    \left|I_{3}\right| \leq C e^{-c^{\prime} \sqrt{s}},\label{3I3}
\end{equation}
where $c'>0$ as $s \to \infty$.

By Lemma \ref{Lemma 3.6} and (\ref{3I1}), multiplying $s$ on the both sides of (\ref{3Z}), we can deduce that
\begin{eqnarray*}
        &&2 v_{j}(\mathbf{0}) \eta(\mathbf{0})\left[\left(\mu\left(\theta_{M}\right)^{-2}-\mu\left(\theta_{M}\right)^{-2} e^{-\sqrt{s h} \mu\left(\theta_{M}\right)}-\mu\left(\theta_{M}\right)^{-1} \sqrt{s h} e^{-\sqrt{s h} \mu\left(\theta_{M}\right)}\right) C_{2}^{+}\right. \\
        &&\left.+\left(\mu\left(\theta_{m}\right)^{-2}-\mu\left(\theta_{m}\right)^{-2} e^{-\sqrt{s h} \mu\left(\theta_{m}\right)}-\mu\left(\theta_{m}\right)^{-1} \sqrt{s h} e^{-\sqrt{s h} \mu\left(\theta_{m}\right)}\right) C_{2}^{-}\right] \\
        &=&s\left[I_{3}-\left(F_{1}(\mathbf{0})+F_{2}(\mathbf{0})+F_{3 j}(\mathbf{0})\right) \int_{S_{h}} u_{0}\left(s \Bx^{\prime}\right) \mathrm{d} \Bx^{\prime}-\Delta_{j}(s)\right. \\
        &&-\eta(\mathbf{0})\left(I_{22}^{+}+I_{22}^{-}\right)-I_{\eta}^{+}-I_{\eta}^{-}-\int_{S_{h}} \delta F_{1}\left(\Bx^{\prime}\right) u_{0}\left(s \Bx^{\prime}\right) \mathrm{d} \Bx^{\prime}-\int_{S_{h}} \delta F_{2}\left(\Bx^{\prime}\right) u_{0}\left(s \Bx^{\prime}\right) \mathrm{d} \Bx^{\prime} \\
        &&\left.-\int_{S_{h}} \delta F_{3 j}\left(\Bx^{\prime}\right) u_{0}\left(s \Bx^{\prime}\right) \mathrm{d} \Bx^{\prime}-v_{j}(\mathbf{0}) \eta(\mathbf{0})\left(I_{21}^{-}+I_{21}^{+}\right)-\xi_{j}^{\pm}(s)\right].
 \end{eqnarray*}
Taking $s=j^\beta,\text{ with }\max \left\{0, \frac{\varrho}{1+\alpha}\right\}<\beta<\Upsilon$ and letting $j\to \infty$, together with the use of (\ref{U0}), (\ref{I1N}),(\ref{3DS}), (\ref{3I3}), Lemmas \ref{Lemma 3.6} and \ref{lemma3.7}, we can deduce that
\begin{equation}
    \lim _{j \rightarrow \infty} \eta(\mathbf{0})\left(C_{2}^{-} \mu\left(\theta_{m}\right)^{-2}+C_{2}^{+} \mu\left(\theta_{M}\right)^{-2}\right) v_{j}(\mathbf{0})=0.\label{lim}
\end{equation}
Moreover, by straightforward calculations, we know that
\begin{equation*}
    \begin{split}
        &C_{2}^{-} \mu\left(\theta_{m}\right)^{-2}+C_{2}^{+} \mu\left(\theta_{M}\right)^{-2} \\
        =&\frac{\left(C_{2}^{+} \cos \theta_{m}+C_{2}^{-} \cos \theta_{M}\right)+\mathbf{i}\left(C_{2}^{+} \sin \theta_{m}+C_{2}^{-} \sin \theta_{M}\right)}{\left(\cos \theta_{m}+\mathbf{i} \sin \theta_{m}\right)\left(\cos \theta_{M}+\mathbf{i} \sin \theta_{M}\right)}
        \end{split}
\end{equation*}
Because of the assumption (\ref{3.5}), one can directly verify that
\begin{equation*}
   \big( \cos \theta_{m}+\cos \theta_{M}\big ) \text { and } \big(\sin \theta_{m}+\sin \theta_{M}\big)
\end{equation*}
can not be zero simultaneously. Without loss of generality, we assume that  $\sin \theta_{m}+\sin \theta_{M} \neq 0$. Then we split our arguments into the following two cases:
\begin{enumerate}
    \item $\sin \theta_{m}+\sin \theta_{M}>0$;\smallskip

    \item $\sin \theta_{m}+\sin \theta_{M}<0$.
\end{enumerate}

For the first case, if $\sin \theta_m$ and $\sin \theta_M$ have the same sign, then from (\ref{C2-}), we know that $C_{2}^{+} \sin \theta_{m}+C_{2}^{-} \sin \theta_{M}\neq 0$ which means that
\begin{equation}
    C_{2}^{-} \mu\left(\theta_{m}\right)^{-2}+C_{2}^{+} \mu\left(\theta_{M}\right)^{-2} \neq 0. \label{3.43}
\end{equation}
If $\sin \theta_m$ and $\sin \theta_M$ have different signs, under the assumption (\ref{3.5}) we know that $\sin \theta_m<0$ and $\sin \theta_M>0$. From (\ref{C2-}), we have
\begin{equation*}
    \begin{split}
        \frac{C(\psi)}{1-(k L)^{2}}\left(\sin \theta_{m}+\left(1-2(k L)^{2}\right) \sin \theta_{M}\right) & \leq C_{2}^{+} \sin \theta_{m}+C_{2}^{-} \sin \theta_{M} \\
        & \leq \frac{C(\psi)}{1-(k L)^{2}}\left(\left(1-2(k L)^{2}\right) \sin \theta_{m}+\sin \theta_{M}\right).
        \end{split}
\end{equation*}
For a given $0<\varepsilon<1$, we can choose an appropriate $L$ such that $0<k L<\sqrt{\varepsilon / 2},$ from which we can derive the bounds as follows
\begin{equation}
    \begin{split}
        \frac{C(\psi)}{1-(k L)^{2}}\left(\sin \theta_{m}+(1-\varepsilon) \sin \theta_{M}\right) & \leq C_{2}^{+} \sin \theta_{m}+C_{2}^{-} \sin \theta_{M} \\
        & \leq \frac{C(\psi)}{1-(k L)^{2}}\left((1-\varepsilon) \sin \theta_{m}+\sin \theta_{M}\right).
        \end{split}\label{3.44}
\end{equation}
Since $\sin \theta_{m}+\sin \theta_{M}>0,$ we can establish the lower bound in (\ref{3.44}) as follows. Denote $\varepsilon_{0}=\min \left\{\frac{\sin \theta_{m}+\sin \theta_{M}}{2 \sin \theta_{M}}, 1\right\}$ and choose $\varepsilon \in\left(0, \varepsilon_{0}\right) .$ It can be verified that
\begin{equation*}
    C_{2}^{+} \sin \theta_{m}+C_{2}^{-} \sin \theta_{M} \geq \frac{C(\psi)}{1-(k L)^{2}}\left(\sin \theta_{m}+(1-\varepsilon) \sin \theta_{M}\right)>0
\end{equation*}
which indicates that (\ref{3.43}) holds as well.

For the second case, if $\sin \theta_m=0$ or $\sin \theta_M=0$ is satisfied, from (\ref{3.44}) we know that
\begin{equation}
    C_{2}^{+} \sin \theta_{m}+C_{2}^{-} \sin \theta_{M}<0\label{3.45}.
\end{equation}
Otherwise, if $\left|\sin \theta_{m}\right| \leq\left|\sin \theta_{M}\right|,$ from the fact that $(1-\varepsilon)\left|\sin \theta_{m}\right| \leq\left|\sin \theta_{M}\right|,$ we know that (\ref{3.45}) still holds from the upper bound of (\ref{3.44}). If $\left|\sin \theta_{m}\right|>\left|\sin \theta_{M}\right|$ we can choose $\varepsilon$ with $\varepsilon>1-\left|\sin \theta_{M}\right| /\left|\sin \theta_{m}\right|>0$ such that (\ref{3.45}) is also fulfilled from the upper bound of (\ref{3.44}). Therefore, for the second case, we know that (\ref{3.43}) is always fulfilled.
Therefore, by (\ref{lim}) and (\ref{3.4}) we know that
\begin{equation*}
    \lim _{j \rightarrow \infty} v_{j}(\mathbf{0})=0.
\end{equation*}
Next, in order to simplify the notations, we define
$$
\kappa:=B((\Bx_{c}, x_n^c ),\rho)\cap (W\times (-M,M)).
$$
Then by using the fact that
\begin{equation*}
    \begin{split}
        &\lim _{\rho \rightarrow+0} \frac{1}{m(\kappa)} \int_{\kappa}|v(\Bx)| \mathrm{d} \Bx\\
         \leq & \lim _{j \rightarrow \infty}\left(\lim _{\rho \rightarrow+0} \frac{1}{m\left(\kappa\right)} \int_{\kappa}\left|v(\Bx)-v_{j}(\Bx)\right| \mathrm{d} \Bx+\lim _{\rho \rightarrow+0} \frac{1}{m\left(\kappa\right)} \int_{\kappa}\left|v_{j}(\Bx)\right| \mathrm{d} \Bx\right),
        \end{split}
\end{equation*}
we finally have
\begin{equation*}
    \lim _{\rho \rightarrow+0} \frac{1}{m\left(\kappa\right)} \int_{\kappa}|v(\Bx)| \mathrm{d} \Bx=0.
\end{equation*}

The proof is complete.

\end{proof}

Similar to Corollary \ref{maincor1}, we consider the vanishing property of the transmission eigenfunctions in the case $\eta\equiv 0$ in three dimensions.

\begin{cor}\label{maincor2}
    Let $v, w \in H^{1}(W \times(-M, M))$ be a pair of eigenfunctions to \eqref{3DE} associated with $\eta \equiv 0, k\in \RR_{+}$ and $W \subset \mathbb{R}^{2}$ being defined in (\ref{2.1}), and $M>0.$ For any fixed $x_3^{c} \in(-M, M)$ and $L>0$ defined in Definition \ref{Definition 3.1}, we suppose that $L$ is sufficiently small such that $\left(x_3^{c}-L, x_3^{c}+L\right) \subset(-M, M)$. Moreover, there exists a sufficiently small neighbourhood $S_{h}$ of $\Bx_{c} \in \mathbb{R}^{2}$ such that $q w \in C^{\alpha}\left(\bar{S}_{h} \times[-M, M]\right)$ for $0<\alpha<1$. If the following conditions are fulfilled:
\begin{enumerate}
\item[(a)] the transmission eigenfunction $v$ can be approximated in $H^{1}\left(S_{h} \times(-M, M)\right) b y$
the Herglotz functions $v_{j}, j=1,2, \ldots,$ with kernels $g_{j}$ satisfying
\begin{equation}
    \left\|v-v_{j}\right\|_{H^{1}\left(S_{h} \times(-M, M)\right)} \leq j^{-\Upsilon}, \quad\left\|g_{j}\right\|_{L^{2}\left(\mathbb{S}^{2}\right)} \leq C j^{\varrho},\label{3newasa}
\end{equation}
for some constants $C,\varrho \text{ and }\Upsilon$ with $C>0, \Upsilon>0 \text { and } \varrho<\alpha\Upsilon/2$;
\item[(b)] the angles $\theta_{m}$ and $\theta_{M}$ of the sector $W$ satisfy
\begin{equation}
    -\pi<\theta_{m}<\theta_{M}<\pi \text { and } \theta_{M}-\theta_{m} \neq \pi;\label{3newasb}
\end{equation}
\end{enumerate}
then it holds that
\begin{equation*}
    \lim _{\rho \rightarrow+0} \frac{1}{m\left(B\left(\mathbf{x}_{c}, \rho\right)\cap W\right)} \int_{B\left(\mathbf{x}_{c}, \rho\right)\cap W} \mathcal{R}(V w)\left(\mathbf{x}^{\prime}\right) \mathrm{d} \mathbf{x}^{\prime}=0,
\end{equation*}
where $q\left(\mathbf{x}^{\prime}, x_3\right)=1+V\left(\mathbf{x}^{\prime}, x_3\right)$.
\end{cor}

\begin{proof}
    Without loss of generality, we assume that $\Bx_c=0$. Since $\eta \equiv 0$, from (\ref{3Z}), (\ref{3I1}), we can obtain that
    \begin{equation}
        \begin{split}
            &\left(F_{1}(\mathbf{0})+F_{2}(\mathbf{0})+F_{3 j}(\mathbf{0})\right) \int_{S_{h}} u_{0}\left(s \Bx^{\prime}\right) \mathrm{d} \Bx^{\prime}+\Delta_{j}(s)\\
           &= I_{3}-\int_{S_{h}} \delta F_{1}\left(\Bx^{\prime}\right) u_{0}\left(s \Bx^{\prime}\right) \mathrm{d} \Bx^{\prime}
        -\int_{S_{h}} \delta F_{2}\left(\Bx^{\prime}\right) u_{0}\left(s \Bx^{\prime}\right) \mathrm{d} \Bx^{\prime}-\int_{S_{h}} \delta F_{3 j}\left(\Bx^{\prime}\right) u_{0}\left(s \Bx^{\prime}\right) \mathrm{d} \Bx^{\prime},
        \end{split}\label{newmain2}
    \end{equation}
    where $\Delta_{j}(s)$ and $I_{3}$ are defined in (\ref{3Djs}) and (\ref{e3I3}), and $\delta F_{1}(\Bx')$, $\delta F_{2}(\Bx')$ and $\delta F_{3j}(\Bx')$ are deduced by Lemma \ref{lemma2.4} with $F_{1}(\Bx')$, $F_{2}(\Bx')$ and $F_{3j}(\Bx')$ defined in (\ref{F123}) and (\ref{F3j}).
    Since $v=w \text { on } \Gamma^{\pm} \times(-M, M)$, it easy to see that
    \begin{equation*}
        F_{1}(\mathbf{0})=\int_{-L}^{L} \psi^{\prime \prime}\left(x_3\right)\left(v\left(\mathbf{0}, x_3\right)-w\left(\mathbf{0}, x_3\right)\right) \mathrm{d} x_3=0.\label{F1}
    \end{equation*}
    Multiplying $s^{2}$ on both sides of (\ref{newmain2}), taking $s=j^{\beta}$ with $\max\{\varrho/\alpha,0\}<\beta<\Upsilon/2$, using the assumptions (\ref{3newasa}) and (\ref{3newasb}), and by letting $j\to \infty$, from (\ref{F1}), (\ref{U0}), (\ref{I1N}), (\ref{3DS}) and (\ref{3I3}), we can prove that
    \begin{equation*}
        \lim _{j \rightarrow \infty} F_{3 j}(\mathbf{0})=-F_{2}(\mathbf{0}),
    \end{equation*}
    which in turn implies that
    \begin{equation*}
        \lim _{j \rightarrow \infty} \mathcal{R}\left(v_{j}\right)(\mathbf{0})=\mathcal{R}(q w)(\mathbf{0}).
    \end{equation*}
    Using the boundary condition in (\ref{3DE}) and Definition \ref{Definition 3.1}, we have that $\mathcal{R}(w)\left(\mathbf{x}^{\prime}\right)=\mathcal{R}(v)\left(\mathbf{x}^{\prime}\right) \text { on } \Gamma$. Hence, we have
    \begin{equation*}
    \begin{split}
     &   \lim _{\rho \rightarrow+0} \frac{1}{m(B(\mathbf{0}, \rho)\cap W)} \int_{B(\mathbf{0}, \rho)\cap W} \mathcal{R}(v)\left(\mathbf{x}^{\prime}\right) \mathrm{d} \mathbf{x}^{\prime}\\
        =& \lim _{\rho \rightarrow+0} \frac{1}{m(B(\mathbf{0}, \rho)\cap W)} \int_{B(\mathbf{0}, \rho)\cap W} \mathcal{R}(w)\left(\mathbf{x}^{\prime}\right) \mathrm{d} \mathbf{x}^{\prime}.
        \end{split}
    \end{equation*}
    which together with the facts that
    \begin{equation*}
        \begin{split}
                \lim _{j \rightarrow \infty} \mathcal{R}\left(v_{j}\right)(\mathbf{0}) &=\lim _{j \rightarrow \infty} \lim _{\rho \rightarrow+0} \frac{1}{m(B(\mathbf{0}, \rho)\cap W)} \int_{B(\mathbf{0}, \rho)\cap W} \mathcal{R}\left(v_{j}\right)\left(\mathbf{x}^{\prime}\right) \mathrm{d} \mathbf{x}^{\prime} \\
                &=\lim _{\rho \rightarrow+0} \frac{1}{m(B(\mathbf{0}, \rho)\cap W)} \int_{B(\mathbf{0}, \rho)\cap W} \mathcal{R}(v)\left(\mathbf{x}^{\prime}\right) \mathrm{d} \mathbf{x}^{\prime}, \\
                \mathcal{R}(q w)(\mathbf{0}) &=\lim _{\rho \rightarrow+0} \frac{1}{m(B(\mathbf{0}, \rho)\cap W)} \int_{B(\mathbf{0}, \rho)\cap W} \mathcal{R}(q w)\left(\mathbf{x}^{\prime}\right) \mathrm{d} \mathbf{x}^{\prime}.
        \end{split}
    \end{equation*}
  readily completes the proof of the corollary.
\end{proof}

\begin{remark}
    If $V(\Bx',x_n)$ is continuous near the edge point $(\Bx_{c},x_3^{c})$ and $V(\Bx_{c},x_3^{c})\neq 0$, by the dominant convergent theorem and Definition \ref{Definition 3.1}, we can prove that
    \begin{equation*}
        \lim _{\rho \rightarrow+0} \frac{1}{m\left(B\left(\mathbf{x}_{c}, \rho\right)\cap W\right)} \int_{B\left(\mathbf{x}_{c}, \rho\right)\cap W} \int_{x_3^{c}-L}^{x_3^{c}+L} \psi\left(x_3\right) w\left(\mathbf{x}^{\prime}, x_3\right) \mathrm{d} \mathbf{x}^{\prime} \mathrm{d} x_3=0.
    \end{equation*}
    Furthermore, if $\psi(x_3^{c})\neq 0$, we can show that
    \begin{equation*}
        \lim _{\rho \rightarrow+0} \frac{1}{m\left(B\left(\mathbf{x}_{c}, \rho\right)\cap W\right)} \int_{B\left(\mathbf{x}_{c}, \rho\right)\cap W} \int_{x_3^{c}-L}^{x_3^{c}+L} w\left(\mathbf{x}^{\prime}, x_3\right) \mathrm{d} \mathbf{x}^{\prime} \mathrm{d} x_3=0.
    \end{equation*}
    which describes the vanishing property of the transmission eigenfunctions near the edge corner in three dimensions.
\end{remark}

\section*{Acknowledgement}
The work of Y. Deng was supported by NSF grant of China No. 11971487 and NSF grant of Hunan No. 2020JJ2038.
 The work of H Liu was supported by a startup fund from City University of Hong Kong and the Hong Kong RGC General Research Fund (projects 12301420, 12302919, 12301218).

\renewcommand\refname{References}

\end{document}